\documentclass[11pt]{amsart}
\usepackage{amsmath}
\usepackage{amssymb}
\usepackage{amsthm}
\usepackage{yfonts}

\usepackage[]{hyperref}

\usepackage{graphicx}
\usepackage{overpic}

\usepackage{enumitem}

\usepackage{subcaption}

\usepackage[margin=2.7cm, marginparwidth=2.4cm]{geometry}

\newtheorem{theorem}{Theorem}[section]
\newtheorem{lemma}[theorem]{Lemma}
\newtheorem{proposition}[theorem]{Proposition}

\newtheorem{claim}[theorem]{Claim}

\newtheorem{convention}[theorem]{Convention}
\newtheorem{remark}[theorem]{Remark}

\theoremstyle{definition} 
\newtheorem{definition}[theorem]{Definition}

\newcounter{theoremalph}

\newtheorem{introtheorem}[theoremalph]{Theorem}
\newtheorem{introcorollary}[theoremalph]{Corollary}

\usepackage{xcolor}

\usepackage{marginnote}
\newcounter{mynote}

\title{Characterizing divergence and thickness in right-angled Coxeter groups}
\author{Ivan Levcovitz}
\date{}
\thanks{The author was supported in part by a Technion fellowship.}
\begin{document}
\maketitle
\begin{abstract}
	We completely classify the possible divergence functions for right-angled Coxeter groups (RACGs). In particular, we show that the divergence of any such group is either polynomial, exponential or infinite.
	We prove that a RACG is strongly thick of order $k$ if and only if its divergence function is a polynomial of degree $k+1$. Moreover, we show that the exact divergence function of a RACG can easily be computed from its defining graph by an invariant we call the hypergraph index.
\end{abstract}
\section{Introduction}
Given a finite simplicial graph $\Gamma$ with vertex set $V(\Gamma)$ and edge set $E(\Gamma)$, 
the corresponding right-angled Coxeter group (RACG for short) is given by the presentation:
\[ \langle s \in V(\Gamma) ~|~ s^2 =1 ~ \text{ for all } s \in V(\Gamma) \text{ and } st = ts \text{ for all } (s, t) \in E(\Gamma) \rangle \] 

In this article, we provide an explicit method to compute the divergence and order of (strong) thickness for all RACGs, and we prove that these quasi-isometry invariants are in fact equivalent in this setting.
Consequently, 
this result establishes an exact connection between 
the largest rate that a pair of geodesic rays can diverge in (the Cayley graph of) a RACG and the coarse complexity of an optimal decomposition of this group into subsets which do not exhibit non-positive curvature (i.e. whose asymptotic cones do not contain cutpoints).

Besides divergence and thickness, there are few computable quasi-isometry invariants that are applicable to non-relatively hyperbolic RACGs.
In fact, non-relatively RACGs with super-quadratic divergence could previously only be distinguished up to quasi-isometry in some exceptional cases.
A consequence of our main theorem is that we can distinguish many more RACGs up to quasi-isometry, as we are able to explicitly determine their divergence functions. As a concrete example, $W_{\Gamma_2}$ and $W_{\Gamma_3}$ are not quasi-isometric (where $\Gamma_2$ and $\Gamma_3$ are as in Figure~\ref{fig:example} below), as by our main theorem $W_{\Gamma_2}$ has divergence a cubic polynomial and $W_{\Gamma_3}$ has divergence a quartic polynomial.

Given a bi-infinite geodesic $\gamma: \mathbb{R} \to X$ in a metric space $X$, its geodesic divergence is the function $\text{Div}^{\gamma}(r)$ whose value is the infimum over the length of $\alpha$, where $\alpha$ is a path from $\gamma(-r)$ to $\gamma(r)$ which does not intersect the open ball based at $\gamma(0)$ of radius $r$. If no such path exists, then we say that $\text{Div}^{\gamma}(r)$ is infinite.
There is a corresponding notion of the divergence function of a metric space which, for $r >0$, roughly takes value the supremum over the lengths of all minimal length paths, which avoid an open ball of radius proportional to $r$ and which connect two points that are distance proportional to $r$ apart.
The divergence function of a finitely generated group is defined to be the divergence function of one of its Cayley graphs, and
it is a quasi-isometry invariant of finitely generated groups up to a usual equivalence of functions used in geometric group theory.
The geodesic divergence of a bi-infinite geodesic in the Cayley graph of a group gives a lower bound on the group's divergence.

The second quasi-isometry invariant treated in this article is strong thickness (which we often simply refer to as ``thickness''). 
A group is thick of order $0$ if and only if all of its asymptotic cones do not contain cutpoints. 
Roughly, a group is thick of order $k$, if it is not thick of order $k-1$, and its Cayley graph coarsely decomposes into thick pieces of order strictly less than $k$. Moreover, given any two such pieces $P$ and $P'$ in this decomposition, there exists a sequence of pieces $P = P_1, \dots, P_m = P'$ in the decomposition such that $P_i$ has infinite-diameter coarse intersection with $P_{i+1}$ for $1 \le i < m$. 
\vspace{.5cm}
\begin{figure}[htbp]
	\begin{overpic}[scale=.4]{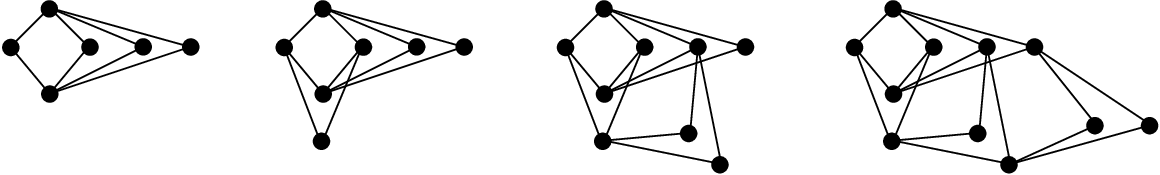}
		\put(5,18){\Tiny $\Gamma_0$}
		\put(28,18){\Tiny $\Gamma_1$}
		\put(53,18){\Tiny $\Gamma_2$}
		\put(78,18){\Tiny $\Gamma_3$}
	\end{overpic}
	\caption{For each $0 \le i \le 3$, the graph $\Gamma_i$ has hypergraph index $i$.}
	\label{fig:example}
\end{figure}

Finally, we will need a third notion: the hypergraph index.
Given a simplicial graph, one can explicitly compute its hypergraph index, which takes value either a non-negative integer or~$\infty$.
Our main theorem stated below, characterizes divergence and thickness in RACGs in terms of the hypergraph~index of their defining graphs. This gives an easy and explicit method of computing thickness and divergence of any RACG.
\begin{introtheorem}\label{intro_thm:char}
	Let $W_\Gamma$ be a RACG and $k \ge 0$ an integer. Then the following are equivalent:
	\begin{enumerate}
		\item \label{main_thm3} The hypergraph index of $\Gamma$ is $k$.
		\item \label{main_thm1} The divergence of $W_\Gamma$ is $r^{k+1}$, and the Cayley graph of $W_\Gamma$ contains a periodic geodesic with geodesic divergence $r^{k+1}$.
		\item \label{main_thm2}  The group $W_\Gamma$ is strongly thick of order $k$.
	\end{enumerate}
\end{introtheorem}  

As we later discuss, there were already some known bounds between divergence, thickness and the hypergraph index \cite{Behrstock-Drutu, Levcovitz-thick}. 
Moreover, the above theorem was known in the special cases where $k=0$ and $k=1$ \cite{Behrstock-Hagen-Sisto, Dani-Thomas, Levcovitz-div} and was conjectured in \cite{Levcovitz-thick}.
However, the proof of those two known cases resisted generalization, and new, significantly more refined methods had to be developed in order to prove our result in its full generality.
Additionally, given an arbitrary non-relatively hyperbolic RACG which is not thick of order $0$ or $1$, its exact divergence function and order thickness was previously unknown except in very specialized cases such as those treated in \cite{Dani-Thomas} and \cite{Levcovitz-div} which involve RACGs whose associated CAT(0) cube complexes contain hyperplanes with very well-behaved separation properties.

As a corollary, we obtain a complete classification of divergence functions in RACGs:
\begin{introcorollary} \label{intro_cor}
	The divergence of a RACG $W_\Gamma$ is either polynomial, exponential or infinite.
\end{introcorollary}
\begin{proof}
	If the hypergraph index of $\Gamma$ is $\infty$, then $W_\Gamma$ is relatively hyperbolic  (see \cite{Levcovitz-thick}) and  has divergence an exponential function if one-ended \cite{Sisto}[Theorem~1.3] and has infinite divergence if it is finite, infinite-ended or two-ended. Otherwise, the hypergraph index of $\Gamma$ is an integer and $W_\Gamma$ has divergence a polynomial function by Theorem~\ref{intro_thm:char}.
\end{proof}

These large gaps exhibited in the divergence function spectrum of RACGs do not exist in arbitrary finitely generated groups: there are groups with ``intermediate'' divergence functions. 
For instance, Olshanskii--Osin--Sapir show there are lacunary hyperbolic groups exhibiting divergence functions which are strictly between linear and quadratic  \cite{Olshanskii-Osin-Sapir}. Additional groups with exotic divergence functions were found by Gruber--Sisto \cite{Gruber-Sisto}. More recently, Brady--Tran amazingly construct finitely-presented groups with divergence functions $r^\alpha$ for a dense set of $\alpha \in [2, \infty]$~\cite{Brady-Tran}.

Theorem \ref{intro_thm:char} can also be utilized in the study of random RACGs, i.e., RACGs defined by a random graph in the Erd\H{o}s--R\'{e}nyi model.
Behrstock--Hagen--Sisto compute an explicit threshold function for when a random RACG is thick or relatively hyperbolic \cite{Behrstock-Hagen-Sisto}.
Building on this work, Behrstock--Falgas-Ravry--Hagen--Susse give threshold functions for the transition in random RACGs between thickness of order $0$, $1$ and larger than $1$ \cite{Behrstock-Falgas-Ravry-Hagen-Susse}.
These authors demonstrate an interesting interplay between random graphs and the coarse geometry of RACGs.
Threshold functions for the transition between orders of thickness larger than $1$ are unknown. This question is now more tractable, as Theorem~\ref{intro_thm:char} reduces it to an analysis of random graphs: determining threshold functions for the transition from one hypergraph index to the next. 

We now discuss some background on divergence, thickness and the hypergraph index. 
Gromov expected that one-ended groups which act geometrically on a CAT(0) space should exhibit either linear or exponential divergence \cite{Gromov}. However, it turns out that many important classes of groups (including CAT(0) ones) do not fall into this dichotomy. For instance, Gersten showed that the divergence of a $3$--manifold group is linear, quadratic or exponential \cite{Gersten}. Additionally, Behrstock--Charney prove that the divergence of a right-angled Artin group is, similarly, linear, quadratic or infinite (when not one-ended) \cite{Behrstock-Charney}. Interestingly, most mapping class groups also exhibit quadratic divergence \cite{Behrstock-mcg, Duchin-Rafi}. 
There are also CAT(0) groups with divergence function a polynomial of any degree, and such groups were first constructed by Macura \cite{Macura} and independently also by Behrstock-Dru\c{t}u \cite{Behrstock-Drutu}.

For each non-negative integer, Dani--Thomas give an example of a RACG whose divergence is a polynomial of this degree \cite{Dani-Thomas}.
Furthermore, these authors give a graph-theoretic characterization of $2$--dimensional RACGs  (i.e., whose defining graph does not contain $3$--cycles) with linear and quadratic divergence. These characterizations were later generalized to RACGs of arbitrary dimension, with the linear case being done by Behrstock--Hagen--Sisto \cite{Behrstock-Hagen-Sisto} and the quadratic case done by the author \cite{Levcovitz-div}.

Thick spaces were first defined by Behrstock--Dru\c{t}u--Mosher in \cite{Behrstock-Drutu-Mosher} where it is shown that the order of thickness is a quasi-isometry invariant and that thick groups are non-relatively hyperbolic. Behrstock--Dru\c{t}u later define a slightly stronger, more quantified, version of thickness, known as strong thickness which appears to be becoming the standard definition. 
Strong thickness is still a quasi-isometry invariant, and all groups known to be thick, are strongly thick of the same order.
Furthermore, these authors show that a group which is strongly thick of order $k$ has divergence function bound above by a polynomial of degree $k+1$ \cite{Behrstock-Drutu}. 
Many well-studied non-relatively hyperbolic groups are strongly thick, and there is often a dichotomy where a group, in a given class of groups, is either strongly thick or hyperbolic relative to strongly thick peripheral subgroups (see for instance \cite{Behrstock-Hagen-Sisto} and \cite{Hagen-free-by-z}).

The hypergraph index was introduced by the author in \cite{Levcovitz-thick}.
The hypergraph index of a RACG is defined to be the hypergraph index of its defining graph, and it was previously known to give some measure of the RACG's coarse complexity.
For instance, the hypergraph index of a RACG is $\infty$ if and only if the RACG is relatively hyperbolic.
Moreover, if a given RACG is quasi-isometric to a right-angled Artin group then its hypergraph index is either $0$, $1$ or $\infty$. 
The hypergraph index was previously only known to be quasi-isometry invariant within the class of $2$--dimensional RACGs, and the proof of this used the structure of quasi-flats.
Finally, the hypergraph index is known to give an upper bound on thickness: a RACG of hypergraph index $k \neq \infty$ is thick of order at most $k$.

Given these known bounds between divergence, thickness and the hypergraph index, 
in order to prove Theorem \ref{intro_thm:char}, a lower bound on the divergence function of a RACG in terms of the hypergraph index must be established. 
This is the content of the following theorem:

\begin{introtheorem} \label{intro_thm:div}
	Let $\Gamma$ be a simplicial graph with hypergraph index $k \neq \infty$. Then the Cayley graph of the RACG $W_\Gamma$ contains a periodic geodesic with geodesic divergence a polynomial of degree~$k+1$.
\end{introtheorem}

The proof of Theorem \ref{intro_thm:div} involves a careful analysis of disk diagrams. 
We first define $L$--fences in a disk diagram over a RACG in Section~\ref{sec:fences}. 
 These inductively defined objects consist of a set of dual curves whose intersection pattern naturally corresponds to a subgraph of hypergraph index $L$ in the RACG's defining graph.
Additionally, dual curves which ``cross'' an $L$--fence are forced to intersect it in a way that is compatible with the associated hypergraph index structure.
In Section~\ref{sec:structured}, we define sequences of ``structured'' dual curves in a disk diagram. These sequences have desirable properties, and we show how to find large enough such sequences.
In Section~\ref{sec:surgery} we define disk diagram surgeries which allow us to insert a disk diagram, which contains a well-behaved path, into another disk diagram.
After establishing these necessary concepts and proving some essential properties about them, 
in Section~\ref{sec:div_bounds} we simultaneously prove two technical propositions which are at the heart of the proof of Theorem~\ref{intro_thm:div}. 
These propositions give lower bounds on the lengths of certain paths, and their hypotheses are designed to be weak enough to allow for the inductive argument to work.
The arguments in this section involve a careful analysis of $L$--fences in disk diagrams and the paths which they connect.
We are also required to perform  a series of disk diagram surgeries.
A technical challenge to surgeries is that pathologies such as bigons and nongons are introduced into the resulting diagram, and they cannot be easily removed without possibly destroying $L$--fence structures already found.
We often then need to ``let bigons be bigons'' and to work around these pathologies. 
Finally, in Section~\ref{sec:main_thms}, we utilize the technical work from the previous section to prove the results from the introduction.

\subsection*{Acknowledgements} 
 	I am extremely grateful to the anonymous referee whose comments greatly improved the exposition and preciseness of the paper.
	I am also thankful to Jason Behrstock for helpful comments.

\section{Preliminaries} \label{sec:preliminaries}
\emph{We establish some of the definitions and notation used throughout the article and provide references for a more extensive background.}

\vspace{.4cm}
Let $X$ be a metric space. Given a point $x \in X$ and a constant $R \ge 0$, we always denote the open $R$--ball about $x$ by $B_x(R)$. Given a subspace $Y \subset X$, we denote the $R$--neighborhood of $Y$ by~$N_R(Y)$.
\subsection{Divergence}
We review the definitions of divergence of a metric space and of a geodesic. 
We refer the reader to \cite{Drutu-Mozes-Sapir} for further background and proofs that various notions of divergence are equivalent under mild hypotheses (such as the metric space being the Cayley graph of a finitely generated group).

Let $(X, d)$ be a metric space, and let $0 < \delta \le 1$ and $\lambda \ge 0$ be constants. 
Given points $x, y, b \in X$ such that $\min \{ d(b, x), d(b, y) \} = n > 0$, we define $\text{div}_{\delta, \lambda}(x, y, b)$ to be the infimum over the lengths of paths from $x$ to $y$ which do not intersect the open ball $B_b(\delta n - \lambda)$. If there is no such path, we set $\text{div}_{\delta, \lambda}(x, y, b) = \infty$.
The \textit{divergence of $X$} is the function $\text{Div}^X_{\delta, \lambda}(r)$ which, for each $r \ge 0$, takes value the supremum of $\text{div}_{\delta, \lambda}(x, y, b)$ over all $x, y, b \in X$ such that $d(x, y) \le r$.

Given a pair of non-decreasing functions $f, g: \mathbb{R}_+ \to \mathbb{R}_+$, we write $f \preceq g$ if for some constant $C \ge 1$ we have that
\[f(r) \le C g(Cr + C) + Cr + C\]
for all $r \in \mathbb{R}_+$. We write $f \asymp g$ if $f \preceq g$ and $g \preceq f$. 

Up to the equivalence relation $\asymp$ and for $\delta \le \frac{1}{2}$ and $\lambda \ge 2$, the divergence function $\text{Div}^X_{\delta, \lambda}(r)$ is a quasi-isometry invariant  when $X$ is restricted to metric spaces which are the Cayley graph of a finitely generated group \cite{Drutu-Mozes-Sapir}[Corollary 3.12]. 
In light of this, we can define the \textit{divergence of a finitely generated group} to be the divergence (with $\delta \le \frac{1}{2}$ and $\lambda \ge 2$) of a Cayley graph of the group with respect to a finite generating set, up to the equivalence relation $\asymp$.
We remark that the divergence of a group is equivalent to~$\infty$ if and only if the group is not one-ended. 

We now describe the notion of the divergence of a geodesic.
Let  $\gamma: \mathbb{R} \to X$ be a bi-infinite geodesic with basepoint $b$ in the metric space $X$. The \textit{geodesic divergence} of $\gamma$ is the the function $\text{Div}^{\gamma}_{\delta, \lambda}(r) = \text{div}_{\delta, \lambda}(\gamma(r), \gamma(-r), b)$. 
It is immediate from the definitions that given a bi-infinite geodesic $\gamma$ in a metric space $X$, we have that $\text{Div}^\gamma_{\delta, \lambda}(r) \preceq \text{Div}^X_{ \delta, \lambda}(r)$. Thus, geodesic divergence gives a lower bound on the divergence of a space. Furthermore, it is not difficult to show, that $\text{Div}^{\gamma}_{\delta, \lambda}(r) \asymp \text{Div}^{\gamma}_{1, 0}(r)$. Thus, when computing geodesic divergence, we can always assume that $\delta = 1$ and $\lambda = 0$.

\subsection{Strongly thick metric spaces}

In this article, we will not directly use the definition of strongly thick spaces, as we are able to apply known results giving relationships between thickness, divergence and the hypergraph index (see Theorem \ref{thm_thick_bounds_div} and Theorem \ref{thm:hi_bounds_thick}). For completeness, we still include the definition here.
We refer the reader to \cite{Behrstock-Drutu} and \cite{Behrstock-Drutu-Mosher} for more detailed backgrounds. 

Let $C, L > 0$ be constants. A subset $Y$ of a metric space $X$ is $(C,L)$--quasi-convex if given any $y, y' \in Y$ there exists an $(L,L)$--quasi-geodesic contained in $N_C(Y)$ from $y$ to $y'$.

A metric space is \textit{strongly $(C,L)$--thick of order $0$} if the following two conditions hold: (1) No asymptotic cone of $X$ contains a cutpoint (equivalently, $\text{Div}_{\delta, \lambda}^X(r)$ is a linear function for every $0 < \delta < \frac{1}{54}$ and $\lambda \ge 0$ \cite{Drutu-Mozes-Sapir});
(2) For each $x \in X$, there exists a bi-infinite $(L, L)$--quasi-geodesic in $X$ which intersects the ball $B_x(C)$.
A metric space that is strongly $(C,L)$--thick of order $0$ for some $C$ and $L$, is also called \textit{wide}.

For each integer $k \ge 1$, a metric space $X$ is \textit{strongly $(C, L)$--thick of order at most $k$} if there is a collection $\mathcal{Y}$ of $(C, L)$--quasi-convex subsets of $X$ 
which are each strongly $(C,L)$--thick of order at most $k-1$ with respect to each of their induced metrics.
Moreover, we have that  $X = \bigcup_{Y \in \mathcal Y}N_C(Y)$, and, additionally, for every $Y, Y' \in \mathcal{Y}$ and every $x \in X$ such that $B_x(3 C) \cap Y \neq \emptyset$ and $B_x(3C) \cap Y' \neq \emptyset$, it follows that there exists a sequence $Y = Y_1, \dots, Y_n = Y'$ of subspaces in $\mathcal Y$, with $n \le L$,
such that for all $1 \le i < n$, $N_C(Y_i) \cap N_C(Y_{i+1})$ has infinite diameter, $N_C(Y_i) \cap N_C(Y_{i+1}) \cap B_x(L) \neq \emptyset$ and $N_L(N_C(Y_i) \cap N_C(Y_{i+1}))$ is path connected.

We say that a metric space is \textit{strongly thick of order $k$} if it is strongly $(C,L)$--thick of order at most $k$ for some $C, L > 0$ and is not strongly $(C', L')$--thick of order $k-1$ for any choices of $C',L' > 0$ and any choice of subspaces. The order of strong thickness is a quasi-isometry invariant (see \cite{Behrstock-Drutu} and~\cite{Behrstock-Drutu-Mosher}). Behrstock-Dru\c{t}u also show that the order of strong thickness gives an upper bound on divergence:
\begin{theorem}[Corollary 4.17 of \cite{Behrstock-Drutu}] \label{thm_thick_bounds_div}
	Let $X$ be a metric space which is strongly thick of order at most $k$, then $\text{Div}^X_{\delta, \lambda}(r) \preceq r^{k+1}$ for all $0 < \delta < \frac{1}{54}$ and all $\lambda \ge 0$.
\end{theorem}

\subsection{Right-angled Coxeter groups}

We refer the reader to \cite{Davis} for the general theory of Coxeter groups and to \cite{Dani-survey} for a survey on RACG results.

Given a RACG $W_{\Gamma}$ and  $w = s_1 \dots s_n$, with each $s_i \in V(\Gamma)$, we say that $w$ is a \textit{word in} $W_{\Gamma}$.
We say that the word $w'$ is \textit{an expression} for the word $w$ if $w$ and $w'$ are equal as group elements of $W_{\Gamma}$. 
Given a word $w = s_1 \dots s_n$, its \textit{length} $|w|$ is $n$. 
We say that a word $w$ is \textit{reduced} if $|w|$ is minimal out of all possible expressions for $w$.  

A RACG $W_\Gamma$ acts geometrically on a CAT(0) cube complex $\Sigma_\Gamma$ known as the \textit{Davis complex}. 
The $1$-skeleton of $\Sigma_\Gamma$ is the Cayley graph of $W_\Gamma$ (with the standard generating set) where bigons are collapsed to single edges. 
The edges of $\Sigma_\Gamma$ are labeled by the generators $V(\Gamma)$. 
Moreover, for $n \ge 2$, there is an $n$--cube in $\Sigma_\Gamma$ spanning any set of $2^n$ edges which is (label-preserving) isomorphic to the Cayley graph of $W_{K}$ where $K$ is a subclique of $\Gamma$. 
We refer the reader to \cite{Davis} and \cite{Wise-riches} for further background on the Davis complex and CAT(0) cube complexes respectively. We only directly utilize CAT(0) cube complexes in the proof of Theorem \ref{intro_thm:div} given in the final section.

\subsection{Disk Diagrams}
A \textit{disk diagram} over a RACG $W_\Gamma$ is square complex $D$, with a fixed planar embedding, whose edges are labeled by vertices of $\Gamma$.
Moreover, given a square in $D$, the label of its edges, read in cyclic order, is $stst$ where $s$ and $t$ are a pair of adjacent vertices of $\Gamma$.
All disk diagrams in this article are over RACGs. We refer the reader \cite{Sageev-thesis} and \cite{Wise-riches} for the general theory of disk diagrams over CAT(0) cube complexes.

A square $[0,1] \times [0,1]$ in the disk diagram $D$ contains two \textit{midcubes}: $\{\frac{1}{2} \} \times [0,1]$ and $[0,1] \times \{\frac{1}{2} \}$.
A \textit{dual curve} $H$ in $D$ is a minimal, non-empty, connected collection of midcubes in $D$ such that given any pair of midcubes $m$ and $m'$ in $D$, whose intersection is contained in an edge of $D$, it follows that $m \in H$ if and only if $m' \in H$.
We say that an edge of $D$ is \textit{dual} to $H$ if $H$ intersects this edge.
The \textit{carrier} $N(H)$ of a dual curve $H$ is the set of all cells in $D$ which the dual curve intersects. 
As opposite sides of squares in $D$ have the same label, every edge dual to a given dual curve has this same label which we call the \textit{type} of the dual curve.
It readily follows by how squares are labeled in $D$ that the types of a pair of intersecting dual curves consist of a pair of distinct adjacent vertices in $V(\Gamma)$. 
We can also deduce that no dual curve contains both mid-cubes of a given square.
We will frequently use these facts throughout.

A \textit{path in $D$} is a sequence $e_1, \dots, e_n$ of edges in the $1$--skeleton of $D$ along with, for each edge, a choice of orientation so that the endpoint of $e_i$ is incident to the startpoint of $e_{i+1}$ for each $1 \le i < n$. Its \emph{label} is $s_1 \dots s_n$ where $s_i$ is the label of $e_i$. Additionally, we also consider a single vertex of $D$ to be a path with empty label.
We say that a path is \textit{reduced} if its label is a reduced word in the corresponding RACG.
A dual curve is \textit{dual} to a path, if it is dual to an edge contained in the path. 
A dual curve is dual to at most one edge of a reduced path (see \cite[Lemma 3.2.14]{Davis} for instance).
When we write $\gamma = \gamma_1 \gamma_2$ is a path, it is understood that $\gamma_1$ and $\gamma_2$ are paths, the endpoint of $\gamma_1$ is the startpoint of $\gamma_2$ and $\gamma$ is the concatenation of $\gamma_1$ and $\gamma_2$.
Finally, we say that a path is \textit{simple} if it is topologically a simple path.

We say that the disk diagram $D$ has \textit{boundary path} $\gamma$ if $\gamma$ is a path in $D$ containing every edge on the boundary of $D$ (with respect to the given planar embedding) and is minimal length out of such possible paths.
Note that the sequence of edges $e_1, \dots, e_n$ associated to a boundary path can repeat edges (i.e., $e_i = e_{i+1}$ for some $i$ is possible), and boundary paths are not necessarily reduced.
The \textit{basepoint} of a disk diagram with boundary path $\gamma$ is defined to be the starting vertex of $\gamma$.
Given any word $w$ in the RACG $W_{\Gamma}$ which represents the identity element of $W_{\Gamma}$, it follows from van Kampen's lemma that there is a disk diagram with boundary path labeled by $w$.
Finally, given a simple closed path $\eta$ in a disk diagram $D$, the \textit{subdiagram} $D' \subset D$ with boundary path $\eta$, is the largest subcomplex of $D$ that is contained in the closure of the bounded component of $\mathbb{R}^2 \setminus \eta$ (recall that, by its planar embedding, $D$ is a subset of $\mathbb{R}^2$).

\subsection{Hypergraph index}
We denote the vertex set and edge set of a graph $\Gamma$ respectively by $V(\Gamma)$ and $E(\Gamma)$. 
Let $T \subset V(\Gamma)$ be a subset of vertices of the graph $\Gamma$. The subgraph of $\Gamma$ \textit{induced} by $T$ is the subgraph whose vertex set is $T$ and whose edges consist of all edges in $\Gamma$ connecting a pair of vertices in $T$. We say that the graph $\Delta$ is a \textit{join} if  $\Delta$ contains two subgraphs $\Delta_1$ and $\Delta_2$ such that $V(\Delta) = V(\Delta_1) \cup V(\Delta_2)$ and every vertex of $\Delta_1$ is adjacent to every vertex of $\Delta_2$. We denote such a join graph by $\Delta = \Delta_1 \star \Delta_2$.

Let $\Delta$ be an induced subgraph of $\Gamma$ which decomposes as the join $\Delta = \Delta_1 \star \Delta_2$. We say that $\Delta$ is a \textit{wide subgraph} if, for each $i \in \{1,2\}$, $\Delta_i$ contains two non-adjacent vertices. Furthermore, we say that $\Delta = \Delta_1 \star \Delta_2$ is a \textit{strip subgraph} if $\Delta_1$ consists of exactly two non-adjacent vertices and $\Delta_2$ is a clique. 
We note that the Cayley graph of the RACG $W_\Delta$ is a wide metric space if $\Delta$ is wide, and $W_\Delta$ is isomorphic to $D_{\infty} \times \mathbb{Z}_2^k$ (which is quasi-isometric to $\mathbb{Z}$) if $\Delta$ is a strip subgraph.

Recall that a \textit{hypergraph} $\Lambda$ is a set of vertices $V(\Lambda)$ and a set of \textit{hyperedges} $\mathcal{E}(\Lambda)$, where a hyperedge is a non-empty subset of $V(\Lambda)$. In particular, a graph is just a hypergraph whose hyperedges each contain exactly two vertices.

Fix now a simplicial graph $\Gamma$.
Let $\Omega$ be the set of all maximal wide subgraphs of $\Gamma$, and let $\Psi$ be the set of all maximal strip subgraphs of $\Gamma$.
Let $\Lambda_0 = \Lambda_0(\Gamma)$ be the hypergraph with vertex set $V(\Gamma)$ and hyperedge set $\{V(\Delta) ~ | ~ \Delta \in \Omega \cup \Psi \}$.

We now define hypergraphs $\Lambda_n = \Lambda_n(\Gamma)$ inductively for integers $n > 0$. 
Suppose that the hypergraph $\Lambda_i$ is defined for some $i$.
First, we define an equivalence class $\equiv_i$ on the hyperedges of $\Lambda_i$: given hyperedges $E, E' \in \mathcal{E}(\Lambda_i)$, $E \equiv_i E'$ if there exist a sequence $E = E_1, \dots, E_n = E'$ of hyperedges in $\mathcal{E}(\Lambda_i)$ such that for each $1 \le i < n$, $E_i \cap E_{i+1}$ contains a pair of distinct vertices which are not adjacent in $\Gamma$.
We now define $\Lambda_{i+1}$. The vertex set of $\Lambda_{i+1}$ is equal to $V(\Gamma)$. 
Furthermore, $E \subset V(\Gamma)$ is a hyperedge of $\Lambda_{i+1}$ if and only if $E =  E_1 \cup \dots \cup E_m$ where $\{E_1, \dots, E_m\}$ is a maximal collection of $\equiv_i$--equivalent hyperedges of $\Lambda_i$.
For each $0 \le i < \infty$, we say that $\Lambda_i$ is the \textit{$i$'th hypergraph associated to $\Gamma$}.

We now define the \textit{hypegraph index} of $\Gamma$. 
Suppose first that $\Omega \neq \emptyset$. Then the hypergraph index of $\Gamma$ is defined to be the smallest integer $k \ge 0$ such that the $k$'th hypergraph, $\Lambda_k$, associated to $\Gamma$ contains a hyperedge $E$ such that $E = V(\Gamma)$. 
If no such $k$ exists, we set the hypergraph index of $\Gamma$ to be $\infty$. 
Additionally, if $\Omega = \emptyset$ we also set the hypergraph index of $\Gamma$ to be $\infty$. 
We refer the reader to \cite{Levcovitz-thick}[Figure 1] for an explicit example of the computation of the hypergraph index. See also Figure~\ref{fig:example} for examples of specific graphs of hypergraph index $0$, $1$, $2$ and $3$.

The hypergraph index of $\Gamma$ is $\infty$ if and only if the RACG $W_\Gamma$ is relatively hyperbolic (see \cite{Levcovitz-thick}). The hypergraph index also gives an upper bound on the order of strong thickness:
\begin{theorem}[Theorem B from \cite{Levcovitz-thick}] \label{thm:hi_bounds_thick}
	Let $W_\Gamma$ be a RACG with hypergraph index $k \neq \infty$, then $W_\Gamma$ is strongly thick of order at most~$k$.
\end{theorem}

\section{Fences in disk diagrams} \label{sec:fences}
\emph{In this section we introduce the notion of $L$--fences in disk diagrams. 
	We show how these objects relate to the hypergraph index in Proposition \ref{prop:hyp_index_of_fences}, and in Proposition \ref{prop:intersects_a_spoke} we show that, in some sense, $L$--fences separate a disk diagram.
	We also define $L$--splitting points, which roughly measure the largest ``height'' of an $L$--fence connecting two given paths. }

\subsection{$L$--fences}
Before defining $L$--fences, we first define spokes which will serve as the building blocks of an $L$--fence.

\begin{definition}[Spoke] \label{def_spoke}
  A \textit{spoke} $\mathcal{S}$ in a disk diagram $D$ is a two-element set $\mathcal{S} = \{H, K\}$ where $H$ and $K$
  are dual curves in $D$ whose types are distinct non-adjacent vertices of $\Gamma$. The \textit{type} of the spoke
  $\{H, K\}$ is the pair $\{s, t\}$ where $s$ and $t$ are the types of $H$ and
  $K$ respectively.
\end{definition}

\begin{remark}
	We remark that given a spoke $\{H, K\}$ in the disk diagram $D$, it follows that $H$ and $K$ are distinct and do not intersect in $D$ (as their types are non-adjacent vertices of $\Gamma$).
\end{remark}

We say that a dual curve $Q$ (resp. a path $\gamma$) \textit{intersects} a spoke $\mathcal{S} = \{H, K\}$, if both $H \cap Q \neq \emptyset$ and $K \cap Q \neq \emptyset$ (resp. both $H \cap \gamma \neq \emptyset$ and $K \cap \gamma \neq \emptyset$). A spoke $\mathcal{S} = \{H, K\}$ \textit{intersects} the spoke $\mathcal{S}' = \{H', K'\}$ if both $H$ intersects $\mathcal{S}'$ and $K$ intersects $\mathcal{S}'$.

An $L$--fence, inductively defined below, is a set of spokes in a disk diagram satisfying a certain intersection pattern.

\begin{definition}[$L$-fence] \label{def_L_fence}
	Let $D$ be a disk diagram, and let $\mathcal{F}$ be a set of spokes in $D$.
	We say that $\mathcal{F}$ is a $0$--fence if $\mathcal{F}$ consists of a single spoke.
	For integers $L \ge 1$, we say that $\mathcal{F}$ is an $L$-fence if there exists a sequence $\mathcal{F}_1, \dots, \mathcal{F}_n$ of subsets of $\mathcal{F}$
  satisfying:
  \begin{enumerate}
  	\item For  each $1 \le i \le n$, $\mathcal{F}_i$ is an $(L-1)$--fence, 
  	\item for each $1 \le i < n$,
  	either there exists a spoke $\{H, K\}$ in $\mathcal{F}_i$ 
  	such that $H$ intersects a spoke in $\mathcal{F}_{i+1}$ and $K$ intersects a
  	(possibly different) spoke in $\mathcal{F}_{i+1}$, or, alternatively,
  	there exists a spoke $\{H', K'\}$ in $\mathcal{F}_{i+1}$ 
  	such that $H'$ intersects a spoke in $\mathcal{F}_{i}$ and $K'$ intersects a (possibly different)
  	spoke in $\mathcal{F}_{i}$, and
  	\item $\bigcup_{i=1}^n \mathcal{F}_i = \mathcal{F}$.
  \end{enumerate}
  We call $\mathcal{F}_1, \dots, \mathcal{F}_n$ a \emph{decomposition} of $\mathcal{F}$.
\end{definition}

\begin{figure}[htbp]
\begin{overpic}[scale=.3]{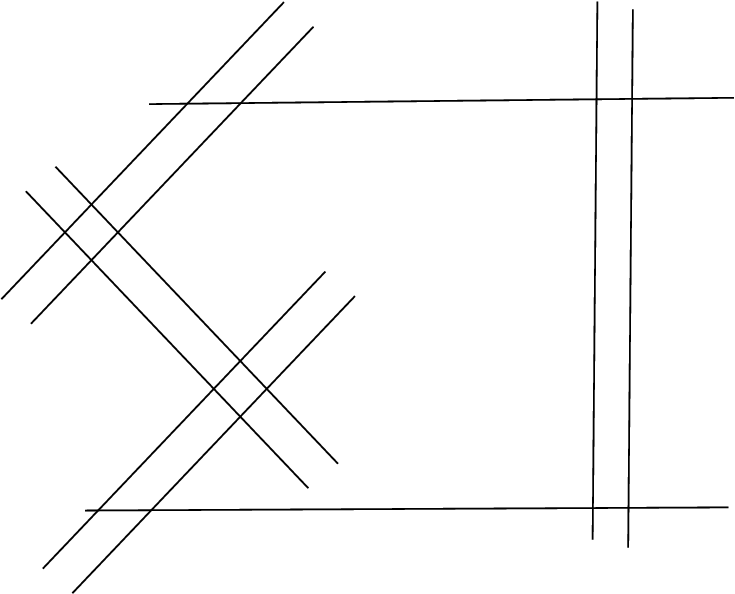}
	\put(-2,0){\Tiny $H_1$}	
	\put(14,0){\Tiny $K_1$}	
	
	\put(15,30){\Tiny $H_2$}	
	\put(24,40){\Tiny $K_2$}	
	
	\put(-8,36){\Tiny $H_3$}	
	\put(0,30){\Tiny $K_3$}
	
	\put(55,5){\Tiny $H_4$}	
	\put(55,70){\Tiny $K_4$}
	
	\put(70,30){\Tiny $H_5$}	
	\put(88,30){\Tiny $K_5$}

\end{overpic}
\caption{We have the spokes $\mathcal{S}_1 = \{H_1, K_1\}$, $\mathcal{S}_2 = \{H_2, K_2\}$, $\mathcal{S}_3 = \{H_3, K_3\}$, $\mathcal{S}_4 = \{H_4, K_4\}$ and $\mathcal{S}_5 = \{H_5, K_5\}$. $\mathcal{F} = \{\mathcal{S}_1, \mathcal{S}_2, \mathcal{S}_3\}$ is a $1$--fence with decomposition  $\mathcal{F}_1 = \{\mathcal{S}_1\}$, $\mathcal{F}_2 = \{\mathcal{S}_2\}$, $\mathcal{F}_3 = \{\mathcal{S}_3\}$.
Additionally, $\mathcal{F}' = \{\mathcal{S}_4, \mathcal{S}_5\}$ is also a $1$--fence with decomposition $\mathcal{F}_1' = \{\mathcal{S}_4\}$, $\mathcal{F}_2' = \{\mathcal{S}_5\}$. 
Finally, $\mathcal{F}'' = \{\mathcal{S}_1, \dots, \mathcal{S}_5\}$ is a $2$--fence with decomposition $\mathcal{F}_1'' = \mathcal{F}$, $\mathcal{F}_2'' = \mathcal{F}'$.}
\label{fig:l_fence}
\end{figure}

\begin{remark}\label{rmk:smaller_fences}
By definition, an $L$--fence is also an $L'$--fence for all $L' \ge L$.
\end{remark}
\begin{remark} \label{rmk:non_uniqueness}
	An $L$--fence admits several different decompositions. For instance, the $2$--fence $\mathcal{F}''$ of Figure~\ref{fig:l_fence} also admits the decomposition $\mathcal{F}_1'' = \mathcal{F}$, $\mathcal{F}_2''= \{\mathcal{S}_4\}$, $\mathcal{F}_3'' = \{\mathcal{S}_5\}$.
\end{remark}

We say that a dual curve \textit{intersects} an $L$--fence, if
 the dual curve intersects a spoke of the $L$--fence.
 A dual curve is \textit{contained in} the $L$--fence $\mathcal{F}$ if it is
  contained in some spoke of $\mathcal{F}$. 
One needs to take care when considering which spokes are contained an $L$--fence:

\begin{remark}
  Let $\mathcal{F}$ be an $L$-fence in a disk diagram $D$ and $\mathcal{S} = \{H, K\}$ a spoke in $D$ such that
  $H$ and $K$ are dual curves contained in $\mathcal{F}$. Then $\mathcal{S}$ may or may not be a spoke of~$\mathcal{F}$. 
\end{remark}

The next lemma gives a way of combining two $L$--fences into a larger $L$--fence.

\begin{lemma} \label{lem:combining_fences}
	Let $L \ge 0$ be an integer, and let $\mathcal{F}$ and $\mathcal{F}'$ be $L$--fences in a disk diagram. 
	If $\mathcal{F} \cap \mathcal{F}' \neq \emptyset$, then 
	$\mathcal{F}'' = \mathcal{F} \cup \mathcal{F}'$ is an $L$--fence.
\end{lemma}
\begin{proof}
	We prove the claim by induction on $L$. If $L = 0$, then $\mathcal{F} = \mathcal{F}' = \{ \mathcal{S} \}$ for some spoke $\mathcal{S}$. 
	The claim then follows trivially.
	
	Suppose now that $L > 0$ and the claim is true for $L-1$.
	Let $\mathcal{F}_1, \dots, \mathcal{F}_n$ and $\mathcal{F}_1', \dots, \mathcal{F}_{n'}'$ be decompositions of $\mathcal{F}$ and $\mathcal{F}'$ respectively.
	Let $\mathcal{S}$ be a spoke in $\mathcal{F} \cap \mathcal{F}'$, and let $j$ and $j'$ be such that $\mathcal{S}$ is contained in $\mathcal{F}_j$ and in $\mathcal{F}_{j'}'$.
	By the induction hypothesis, $\mathcal{T} = \mathcal{F}_j \cup \mathcal{F}_j'$ is an $(L-1)$--fence.
	It now readily follows that $\mathcal{F}'' = \mathcal{F} \cup \mathcal{F}'$ is an $L$--fence with decomposition:
	\[\mathcal{F}_1, \dots, \mathcal{F}_{j-1}, \mathcal{T}, 
	\mathcal F_{j' +1}', \dots, \mathcal{F}_{n'}', \dots \mathcal{F}_{j'+1}', \mathcal{T}, 
	\mathcal{F}_{j'-1}', \dots, \mathcal{F}_1', \dots, \mathcal{F}_{j'-1}', \mathcal{T},
	\mathcal{F}_{j+1}, \dots, \mathcal{F}_n \]
\end{proof}

An $L$--fence $\mathcal{F}$ in a disk diagram $D$ is \textit{maximal}, if 
every $L$--fence containing $\mathcal{F}$ is equal to $\mathcal{F}$.
By the previous lemma, there is a unique maximal $L$--fence containing any given spoke.

Given an $L$--fence $\mathcal{F}$, we define $V(\mathcal{F}) \subset V(\Gamma)$ to be the set of
all vertices $s \in V(\Gamma)$ which are the type of some dual curve in $\mathcal{F}$.
An $L$--fence, together with a collection of dual curves intersecting it, naturally corresponds to a subgraph of $\Gamma$ of hypergraph index at most $L$:

\begin{proposition} \label{prop:hyp_index_of_fences}
	Let $D$ be a disk diagram over the RACG $W_\Gamma$.
	Let $\mathcal{F}$ be an $L$--fence in $D$. 
	Let $Q_1,\dots, Q_m$ be a (possibly empty) sequence of dual curves in $D$
  of types respectively $q_1, \dots, q_m$ such that $Q_i$ intersects $\mathcal{F}$ for each $1 \le
  i \le m$. 
  Then the subgraph of $\Gamma$ induced by $V(\mathcal{F}) \bigcup \{q_1, \dots, q_m\}$ is either a strip subgraph or has hypergraph index at most $L$. 
\end{proposition}
\begin{proof}
	We prove the lemma by induction on $L$.
	We first suppose that $L = 0$.
	In this case, $\mathcal{F}$ contains a single spoke of type $\{s, t\}$ for some non-adjacent vertices $s, t \in \Gamma$.
	It follows that the vertices $q_1, \dots, q_m$ are each adjacent to both $s$ and $t$ in $\Gamma$. 
	Thus, the vertices $\{q_1, \dots, q_m, s, t\}$ either induce a strip subgraph of $\Gamma$ (if $q_1, \dots, q_m$ is empty or spans a clique in $\Gamma$) or a wide subgraph of $\Gamma$ otherwise (thus, one with hypergraph index $0$). 
	This completes the base case.
	
	We now fix $L \ge 1$ and assume by induction that the claim is true for all
  $L' < L$. 
  Let $\mathcal{F}$ be an $L$--fence with decomposition $\mathcal{F}_1 \dots,  \mathcal{F}_n$.
  For each $1 \le i \le n$, let $V_i$ be the subset of vertices of $\Gamma$ consisting of
  $V(\mathcal{F}_i)$, the set of all $q_j \in \{q_1, \dots, q_m\}$ such that $Q_j$ intersects  $\mathcal{F}_i$ and, additionally, all vertices $s \in \Gamma$ such that there exists a dual curve $Q$ contained in $\mathcal{F}$ and of type $s$ which intersects $\mathcal{F}_i$.
  By the induction hypothesis, $V_i$ is either a strip subgraph or has hypergraph index at most $L-1$.
  Furthermore, we have that $V(\mathcal{F}) \bigcup \{q_1, \dots, q_m\}= \bigcup_{i=1}^{n} V_i$.

  Fix $1 \le i < n$.
  Suppose first that there exists a spoke $\mathcal{S} = \{H, K\}$, of type $\{h, k\}$, contained in
  $\mathcal{F}_i$ such that both $H$ and $K$ intersect $\mathcal{F}_{i+1}$.
  It follows that $h$ and $k$ is a pair of non-adjacent vertices of $\Gamma$ contained in $V_i \cap V_{i+1}$.
  On the other hand, if there is no such spoke $\mathcal{S}$, then by the definition of an $L$--fence there must exist
  a spoke of $\mathcal{S}' = \{H', K'\}$, of type $\{h', k'\}$, contained in $\mathcal{F}_{i+1}$ such that both $H'$ and $K'$ intersect $\mathcal{F}_i$. 
  In this case, $h'$ and $k'$ are non-adjacent vertices of $\Gamma$ both in $V_i \cap V_{i+1}$.
  Thus, $V(\mathcal{F}) \bigcup \{q_1, \dots, q_m\}= \bigcup_{i=1}^{n} V_i$
   has hypergraph index at most $L$.
\end{proof}

\begin{definition}[$L$--fence connected paths]
	Let $D$ be a disk diagram containing an $L$--fence $\mathcal{F}$ and paths $\gamma$ and $\gamma'$.
	We say that $\mathcal{F}$ is \textit{intersects} $\gamma$, if some spoke of $\mathcal{F}$ intersects $\gamma$.	
	We say that the path $\gamma'$ is \textit{$L$--fence connected} to the path $\gamma$ if some $L$--fence $\mathcal{F}'$ intersects both $\gamma$ and $\gamma'$. In this case, we also say that $\mathcal{F}'$ \textit{connects} $\gamma$ and $\gamma'$.
\end{definition}

It will often be the case that an $L$--fence $\mathcal{F}$ connects two paths, both of which are on the boundary of a disk diagram. 
Moreover, it will follow that in some sense $\mathcal{F}$ separates this disk diagram and any dual curve ``crossing'' $\mathcal{F}$ must intersect a spoke of $\mathcal{F}$. This is made precise in Proposition \ref{prop:intersects_a_spoke}. Before proving that proposition, we prove a lemma  showing that $L$--fences exhibit a certain connectivity property.

\begin{lemma} \label{lem:connectivity}
	Let $\mathcal{F} = \{\mathcal{S}_1, \dots, \mathcal{S}_k\}$ be an $L$--fence in a disk diagram $D$.
	Let $Y = H_1 \cup \dots \cup H_k \subset D$ be such that, for each $1 \le i \le k$, $H_i$ is a dual curve in $\mathcal{S}_i$. Then $Y$ is connected.
\end{lemma}
\begin{proof}
	The proof will be by induction on $L$. If $L=0$, then $Y$ consists of a single dual curve and so is connected. 
	Suppose now that $L > 0$ and that the statement is true for all $L' < L$. 
	Let $\mathcal{F}_1, \dots, \mathcal{F}_n$ be a decomposition of $\mathcal{F}$.	
	For each $1 \le i \le n$, we have that $\mathcal{F}_i = \{\mathcal{S}_{i_1} \dots \mathcal{S}_{i_{m_i}} \} \subset \mathcal{F}$. Set $Y_i := H_{i_1} \cup \dots \cup H_{i_{m_i}}$. By the induction hypothesis, $Y_i$ is connected.

	Fix $1 \le i < n$.
	Suppose that some spoke $\mathcal{S} = \{H, K\} \in \mathcal{F}_i$ is such that $H$ intersects
	a spoke in $\mathcal{F}_{i+1}$ and $K$ intersects a, possibly different, spoke in $\mathcal{F}_{i+1}$. 
	Up to relabeling $H$ and $K$, we can assume that $H \subset Y_i$.
	Let $\{H',K'\} \in \mathcal{F}_{i+1}$ be such that $H$ intersects $\{H', K'\}$.
	Up to relabeling $H'$ and $K'$, we can assume that $H' \subset Y_{i+1}$.
	As $H \cap H' \neq \emptyset$, it follows that $Y_i \cap Y_{i+1} \neq \emptyset$.
	On the other hand, if such a spoke $\mathcal{S}$ does not exists, then by the definition of an $L$--fence it follows that some spoke of $\mathcal F_{i+1}$ has each of its dual curves intersecting (possibly different) spokes of $\mathcal{F}_i$, and by a similar argument we still deduce that $Y_i \cap Y_{i+1} \neq \emptyset$. 
	Thus, $Y = Y_1 \cup \dots \cup Y_n$ is connected, as $Y_i \cap Y_{i+1} \neq \emptyset$ for all $1 \le i < n$.
\end{proof}

\begin{proposition} \label{prop:intersects_a_spoke}
	Let $D$ be a disk diagram with boundary path $\gamma \eta
  \gamma' \eta'$ such that $\gamma \cap \eta$, $\gamma' \cap \eta$, $\gamma \cap \eta'$ and $\gamma' \cap \eta'$ all consist of a single vertex.
  Let $\mathcal{F}$ be an $L$--fence connecting $\gamma$ and $\gamma'$.
  Then any dual curve that is dual to both $\eta$ and $\eta'$ intersects a spoke of $\mathcal{F}$.
\end{proposition}
\begin{proof}
  The proof will be by contradiction. Let $Q$ be a dual curve, dual to both $\eta$ and $\eta'$. Let $\mathcal{F} = \{\mathcal{S}_1, \dots, \mathcal{S}_k\}$.
  We assume, for a contradiction, that $Q$ does not intersect any spoke of $\mathcal{F}$.
  Consequently,
  there is a set $\{H_1, \dots, H_k\}$ of dual curves such that for each $1 \le i \le k$, $H_i \in \mathcal{S}_i$ and $Q$ does not intersect $H_i$. Let $Y = H_1 \cup \dots \cup H_k$.
  Note that $Y$ is connected by Lemma~\ref{lem:connectivity} and that $Q \cap Y = \emptyset$.
  
  As $\mathcal{F}$ connects $\gamma$ to $\gamma'$, there are spokes $\{H, K\}$ and $\{H', K'\}$ of $\mathcal{F}$ which intersect $\gamma$ and $\gamma'$ respectively. 
  Thus, $Y$ contains a point $p$ of $\gamma$ and a point $p'$ of $\gamma'$. 
  Let $\zeta$ be a path in $Y$ from $p$ to $p'$. 
  By the structure of dual curves in a disk diagram, $\zeta$ can be chosen to not intersect $\eta$ or $\eta'$.
  As $\zeta$ separates $\eta$ from $\eta'$ in $D$, it follows that
  $Q$ intersects $\zeta \subset Y$, a contradiction.
\end{proof}

\subsection{Splitting Points}

In this subsection we define $L$--splitting points. Intuitively, an $L$--splitting point is the first point along an oriented path $\gamma_1$ such that the subpath of $\gamma_1$ after this point is not $L$--fence connected to another given path $\gamma_2$.

\begin{figure}[htbp]
	\begin{overpic}[scale=.4]{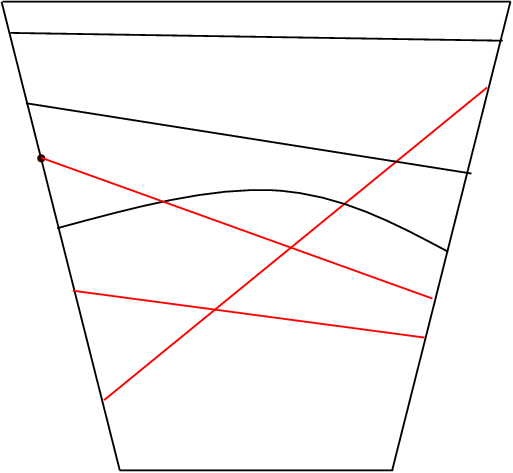}
	 	\put(30,18){\tiny$H_1$}	
	 	\put(20,35){\tiny$H_2$}	
	 	\put(15,59){\tiny$H_3$}	
	 	\put(15,86){\tiny$K_1$}	
	 	\put(15,71){\tiny$K_3$}	
	 	\put(45,57){\tiny$K_2$}			
		\put(-7,50){\tiny$\gamma_1$}
		\put(47,-8){\tiny$\gamma$}
		\put(95,50){\tiny$\gamma_2$}
		\put(47,95){\tiny$\alpha$}
		\put(2,61){\tiny$x$}
	\end{overpic}
	\caption{In the above example, the spokes which intersect $\gamma_1$ and are in a $0$-fence connecting $\gamma_1$ and $\gamma_2$ consist of $\{H_1, K_1\}$, $\{H_2, K_2\}$ and $\{H_3, K_3\}$. The $0$--splitting point $x$ is shown. In particular, $\{K_1, K_3\}$ cannot be a spoke.}
	\label{fig:splitting_pt}
\end{figure}

\begin{definition}[$L$--splitting point]
	Let $D$ be a disk diagram over a RACG with boundary path $\gamma_1 \alpha \gamma_2^{-1} \gamma$ such that 
	$\gamma$, $\gamma_1$ and $\gamma_2$  are reduced. Suppose additionally that $\gamma_1 \cap \gamma$ and $\gamma_2 \cap \gamma$ are each a single vertex and that $\gamma_1 \cap \gamma_2 = \emptyset$. 
	
	Let $L \ge 0$ be an integer. 
	Consider the (possibly empty) set 
	$\mathcal{R} = \{\{H_1,K_1\}, \dots, \{H_n,K_n\}\}$
	of spokes  in $D$ which both intersect $\gamma_1$ and are contained in an $L$--fence connecting $\gamma_1$ and $\gamma_2$. 
	By possibly relabeling, we can suppose that 
	for each $1 \le i \le n$, $H_i \cap \gamma_1$ occurs before $K_i \cap \gamma_1$ along the orientation of $\gamma_1$ and that
	for each $1 \le i < n$, either $H_i = H_{i+1}$ or $H_i \cap \gamma_1$ occurs before $H_{i+1} \cap \gamma_1$ with respect to the orientation of $\gamma_1$.
	We define the \textit{$L$--splitting point} of $(\gamma_1, \gamma_2; \gamma, \alpha)$ to be the point $H_n \cap \gamma_1$ if $\mathcal{R}$ is not empty and to be the starting point of $\gamma_1$ if $\mathcal{R}$ is empty.
\end{definition}

\begin{definition}[Initial and terminal paths with respect to an $L$--splitting point]
	Fix the notation as in the previous definition. 
	Let $x$ be the $L$--splitting point of $(\gamma_1, \gamma_2; \gamma, \alpha)$. Suppose first that $x$ is not equal to the starting point of $\gamma_1$.  
	Let $e$ be the edge of $\gamma_1$ whose midpoint is $x$.
	Let $\gamma_1'$ be the initial subpath of $\gamma_1$ from the starting point of $\gamma_1$ up to, and not including, $e$.
	Let $\gamma_1''$ be the subpath of $\gamma_1$ from $e$ to the endpoint of $\gamma_1$ which does not include $e$.
	On the other hand, if $x$ is the starting point of $\gamma_1$, then we define $\gamma_1' = x$ (a length $0$ path) and $\gamma_1'' = \gamma_1$.
	In either case, we say that $\gamma_1'$ and $\gamma_1''$ are, respectively, the \textit{initial and terminal paths of $\gamma_1$ with respect to the $L$--splitting point $x$}. 
\end{definition}

\begin{remark} \label{rmk:splitting_points}
	With the notation as in the previous two definitions, it is immediate that no $L$--fence in $D'$ connects $\gamma_1''$ and $\gamma_2$.
	Additionally, if $x$ is not equal to the starting point of $\gamma_1$, then there exists an $L$--fence in $D'$ connecting $\gamma_1 \setminus \gamma_1'' = \gamma_1' \cup e$ and $\gamma_2$. 
\end{remark}

\section{Structured sequences of dual curves} \label{sec:structured}

\emph{
Given a path $\gamma$ in a disk diagram, we will often need to find a sequence of dual curves, intersecting $\gamma$ that satisfy certain desirable properties. For instance, we will want these dual curves to be pairwise non-intersecting and to naturally correspond to spokes whose dual curves intersect $\gamma$ close to one another. 
In this section, we define such sequences of dual curves and prove we can find them in different settings.}

\begin{definition}[Structured sequence] \label{def:good_sequence}
	Let $D$ be a disk diagram over the RACG $W_\Gamma$.
	Let $\gamma$ be an oriented path in $D$. 
	We say that a sequence of dual curves $H_1, K_1, \dots, H_n, K_n$, each intersecting $\gamma$, is \textit{structured with respect to $\gamma$} if: 
	\begin{enumerate}
		\item The dual curves $H_1, K_1, \dots, H_n, K_n$ are ordered with respect to the orientation of $\gamma$. More precisely, for all $1 \le i \le n$, $H_i \cap \gamma$ occurs before $K_i \cap \gamma$ with respect to the orientation of $\gamma$ and for all $1 \le i < n$, $K_i \cap \gamma$ occurs before $H_{i+1} \cap \gamma$ with respect to the orientation of $\gamma$.
		
		\item There are non-adjacent vertices $s, t \in \Gamma$ such that, for all $1 \le i \le n$, $H_i$ and $K_i$ are of types $s$ and $t$ respectively.
	\end{enumerate}
	Similarly, we say that a sequence of spokes $\{H_1, K_1\}, \dots, \{H_n, K_n\}$ is structured with respect to $\gamma$ if the corresponding sequence of dual curves $H_1, K_1, \dots, H_n, K_n$ is structured with respect to $\gamma$. 
	Note that a sequence of dual curves is structured with respect to $\gamma$ if and only if the corresponding sequence of spokes is structured with respect to $\gamma$.
	
	If $\alpha$ is another path in $D$, then the sequence $H_1, K_1, \dots, H_n, K_n$  (resp. $\{H_1, K_1\}, \dots, \{H_n, K_n\}$)
	\textit{is structured with respect to $(\gamma, \alpha)$} 
	if it is structured with respect to $\gamma$ and, moreover, both $H_i$ and $K_i$ intersect $\alpha$ for all $1 \le i \le n$. 

	We say that a sequence $H_1, K_1, \dots, H_n, K_n$ (resp. $\{H_1, K_1\}, \dots, \{H_n, K_n\}$) structured with respect to $\gamma$ is \textit{tight} if for each $1 \le i \le n$, the smallest subpath of $\gamma$ containing both $H_i \cap \gamma$ and $K_i \cap \gamma$ is of length at most $|V(\Gamma)|$. 
\end{definition}

By (2) in the definition above, it follows that the dual curves, in a sequence of dual curves structured with respect to a path, are pairwise non-intersecting.
We will use this observation freely throughout.

Given a reduced path $\gamma$, the next lemma guarantees we can always find a tight sequence of dual curves structured with respect to $\gamma$ of size proportional to $|\gamma|$. 

\begin{lemma} \label{lem:nice_sequence}
	Let $\Gamma$ be a non-clique graph. 
	Let $\gamma$ be a reduced path in the disk diagram $D$ over the RACG $W_\Gamma$. 
	Then there is a tight sequence $H_1, K_1, \dots, H_n, K_n$ of dual curves structured with respect to $\gamma$ such that $n \ge  \frac{1}{|V(\Gamma)|^2} \Big\lfloor \frac{|\gamma|}{|V(\Gamma)|} \Big \rfloor$.
\end{lemma}
\begin{proof}
	Set $M := |V(\Gamma)|$. We partition $\gamma = \gamma_1 \dots \gamma_{m+1}$ such that, for each $1 \le i \le m$, $|\gamma_i| = M$ and $|\gamma_{m+1}| < M$.
	Note that $m = \big \lfloor \frac{|\gamma|}{M} \big \rfloor$. 
	For each $1 \le i \le m$, let $w_i = s_{i_1} \dots s_{i_M}$ be the label of $\gamma_i$. 
	
	We claim that, for each $1 \le i \le m$, $\{s_{i_1}, \dots, s_{i_M} \}$ contains a pair of distinct non-adjacent vertices of $\Gamma$. 
	For, suppose otherwise, that $s_{i_1}, \dots, s_{i_M}$ are all vertices in a common clique of $\Gamma$. 
	As $|V(\Gamma)| = M$ and as $\Gamma$ is not a clique, it follows that for some $1 \le j < j' \le M$ we have that $s_{i_j}$ and $s_{i_{j'}}$ are equal as vertices of $\Gamma$. 
	However, it also then follows that $w_i$ is not reduced, contradicting the fact that $\gamma$ is a reduced path. 
	This shows our claim.
	
	Thus, there exists a sequence of dual curves $P_1, Q_1, \dots, P_m, Q_m$ such that for each $1 \le i \le m$, $P_i$ and $Q_i$ both intersect $\gamma_i$ and the types $s_i$ and $t_i$, of $P_i$ and $Q_i$ respectively, are not adjacent in $\Gamma$.
	Moreover, the smallest subpath of $\gamma$ containing both $P_i \cap \gamma$ and $Q_i \cap \gamma$ has length at most $|V(\Gamma)|$ as both $P_i$ and $Q_i$ intersect $\gamma_i$.
	As there are at most $M (M -1) \le M^2$ possible pairs of non-adjacent vertices of $\Gamma$, by the pigeonhole principle there exist non-adjacent vertices $s, t \in \Gamma$ such that at least $\frac{m}{M^2}$ of the spokes $\{P_i, Q_i\}$ are of type $(s, t)$. Thus, there exists a subsequence $H_1 = P_{i_1}, K_1 = Q_{i_1}, \dots, H_n = P_{i_n}, K_n = Q_{i_n}$ which is structured with respect to $\gamma$ such that $n \ge \frac{1}{M^2} \big \lfloor \frac{|\gamma|}{M} \big \rfloor$ and which is tight.
\end{proof}

We will often need for ``enough'' spokes, each satisfying some property, to intersect a path. The following definition makes this notion precise.

\begin{definition}[$M$--adequate sets]
	Let $\gamma$ be a path in a disk diagram $D$. 
	
	A subset $\mathcal{P}(\gamma)$ of edges in $\gamma$ is called an \emph{edge set}.
	We say that $\mathcal{P}(\gamma)$ is \textit{$M$--adequate} if given any sequence of dual curves structured with respect to $\gamma$,  all but possibly $M$ of these dual curves are dual to an edge in $\mathcal{P}(\gamma)$.
	
	An \emph{edge-pair set} $\mathcal{R}(\gamma)$ is a set of unordered pairs $\{e, f\}$  such that $e$ and $f$ are edges of $\gamma$.
	We say that $\mathcal{R}(\gamma)$ is \textit{$M$--adequate} if given any sequence of spokes $\{H_1, K_1\}, \dots, \{H_t, K_t\}$ structured with respect to $\gamma$, then for all but possibly $M$ values of $1 \le i \le t$, there is a pair $\{e, f\} \in \mathcal{R}(\gamma)$ such that $H_i$ is dual to $e$ and $K_i$ is dual to $f$.
	
	When we write that $\mathcal{P}(\gamma)$ (resp. $\mathcal{R}(\gamma)$) is an edge (resp. edge-pair) set, it should be understood that these edges (resp. edge-pairs) are contained in the path $\gamma$.
	
	We say that a \emph{spoke $\{H, K\}$ is dual to the edge-pair $\{e_1, e_2\}$} if $H$ is dual to $e_1$ and $K$ is dual to $e_2$. Let $\mathcal{R}$ be a set of edge-pairs in $D$. We say that the \emph{spoke $\{H, K\}$ is in $\mathcal{R}$}, if it is dual to some edge-pair in $\mathcal{R}$.
\end{definition}

The next lemma, which will be heavily used in Section \ref{sec:div_bounds}, guarantees that, under the right hypotheses, we can find a tight sequence of dual curves structured with respect to~$(\gamma, \alpha)$. 

\begin{figure}[htbp]
	\begin{overpic}[scale=.6]{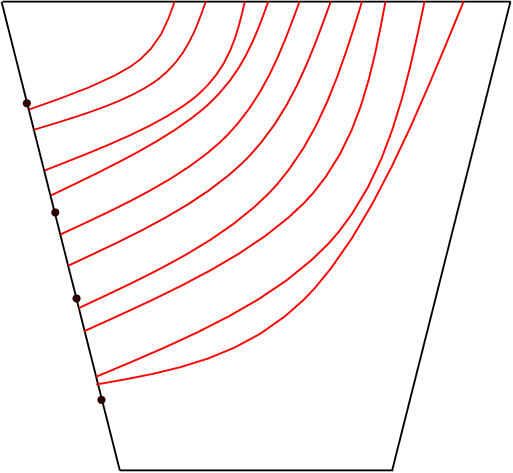}
		\put(50,-4){\Tiny$\beta$}
		\put(16,6){\Tiny$\psi$}
		\put(10,25){\Tiny$\mu_0$}
		\put(7,40){\Tiny$\mu_1$}
		\put(2,60){\Tiny$\mu_2$}
		\put(-1,80){\Tiny$\theta$}
		\put(50,93){\Tiny$\alpha$}
		\put(88,40){\Tiny$\eta$}
		
		\put(40,20){\Tiny$H_1$}
		\put(40,32){\Tiny$K_1$}

		\put(30,32){\Tiny$H_2$}
		\put(30,42){\Tiny$K_2$}

		\put(17,39.5){\Tiny$H_3$}
		\put(20,52){\Tiny$K_3$}
		
		\put(35,66){\Tiny$H_4$}
		\put(34,75){\Tiny$K_4$}
		
		\put(10,65){\Tiny$H_5$}
		\put(10,75.5){\Tiny$K_5$}
	\end{overpic}
	\caption{Conclusion of Lemma~\ref{lem:enough_curves} when $k=2$ and $n=5$. Note that each spoke intersects a single $\mu_j$.}
	\label{fig:enough_spokes}
\end{figure}

\begin{lemma} \label{lem:enough_curves}
	Let $\Gamma$ be a non-clique graph. 
	Let $D$ be a disk diagram over the RACG $W_\Gamma$ with boundary path $\psi \mu_0 \dots \mu_k \theta  \alpha \eta \beta$ such that $\psi \mu_0 \dots \mu_k$ is a reduced path and no dual curve intersects both $\psi \mu_0 \dots \mu_k$ and $\theta$.
	Let $M \ge \max{ \{ |V(\Gamma)|, |\beta| \} }$ be an integer. 
	Suppose that for each $0 \le i \le k$, there is an $M$-adequate edge-pair set $\mathcal{R}(\mu_i)$ 
	such that no spoke in $\mathcal{R}(\mu_i)$ intersects $\eta$.
	Then there exists a tight sequence of spokes $\{H_1, K_1 \}, \dots, \{H_n, K_n\}$ structured with respect to $(\mu_0 \dots \mu_k, \alpha)$ such that:
	\begin{enumerate}
		\item 
		$n \ge \frac{1}{M^2} \big \lfloor \frac{|\mu_0 \dots \mu_k|}{M} \big \rfloor - 3(k+1)M$
		
		\item
		For each $1 \le i \le n$, there exists some $0 \le j \le k$ such that $\{H_i, K_i\}$ is in $\mathcal{R}(\mu_j)$.
	\end{enumerate}
\end{lemma}
\begin{proof}
	By Lemma \ref{lem:nice_sequence}, there is a tight sequence of spokes $\{P_1, Q_1\}, \dots, \{P_r, Q_r\}$
	 structured with respect to $\mu_0 \dots \mu_k$ with $r \ge \frac{1}{M^2} \big \lfloor \frac{|\mu_0 \dots \mu_k|}{M} \big \rfloor$.

	As the dual curves  $P_1, Q_1, \dots, P_r, Q_r$  are pairwise non-intersecting (they are structured), for all but possibly $k$ values of $i \in \{1, \dots, r\}$ there exists some $0 \le j \le k$ such that $\{P_i,Q_i\}$ intersects~$\mu_j$.
	Each of the dual curves $P_1, Q_1, \dots, P_r, Q_r$ intersects $\alpha \eta \beta$ as none of these dual curves are dual to two edges of $\psi \mu_0 \dots \mu_k$ (as it is reduced) and none  intersect $\theta$ by hypothesis.
	Thus, for all but possibly $|\beta| + 1$ values of $i \in \{1, \dots, r\}$ we have that $\{H_i, K_i\}$ intersects either $\alpha$ or $\eta$.
	We conclude that there exists a subsequence $\{P_{i_1}, Q_{i_1}\}, \dots, \{P_{i_{r'}}, Q_{i_{r'}} \}$ of length $r' \ge r - k - |\beta| - 1$ such that, for each $1 \le l \le r'$, $\{P_{i_l}, Q_{i_l} \}$ intersects $\mu_j$ for some distinct $0 \le j \le k$ and $\{P_{i_l}, Q_{i_l} \}$ intersects either $\alpha$ or $\eta$.
	
	For each $0 \le j \le k$, let $A_j \subset \{\{P_{i_1}, Q_{i_1}\}, \dots, \{P_{i_{r'}}, Q_{i_{r'}}\}\}$ be the subset of the spokes which intersect $\mu_j$.
	As $\mathcal{R}(\mu_j)$ is $M$--adequate,
	there exists a subset $A_j' \subset A_j$ of size at least $|A_j| - M$ of spokes in $\mathcal{R}(\mu_j)$.

	Note that $\{\{P_{i_1}, Q_{i_1}\}, \dots, \{P_{i_{r'}}, Q_{i_{r'}} \} \} = A_0 \cup \dots \cup A_{k}$ and that $|A_0' \cup   \dots \cup A_{r'}'| \ge  |A_0 \cup \dots \cup A_{k}| - (k+1)M$.
	It follows that there exists a subsequence $\{H_1, K_1\}, \dots, \{H_n, K_n\}$ of $\{P_{i_1}, Q_{i_1}\}, \dots, \{P_{i_{r'}}, Q_{i_{r'}} \}$ of length $n \ge r' - (k+1)M$ 
	which is structured with respect to $(\mu_0 \dots \mu_k, \alpha)$ and satisfies (2) above. 	
	We also get the bound:
	\[n \ge r' - (k+1)M \ge r - k - |\beta| - 1 - (k+1)M \ge  \frac{1}{M^2} \Big \lfloor \frac{|\mu_0 \dots \mu_k|}{M} \Big \rfloor - 3(k+1)M \]
\end{proof}

\section{Disk diagram surgery} \label{sec:surgery}
\emph{In this section we discuss disk diagram surgery, an operation which allows us to ``insert'' a path into a disk diagram in place of another. We prove two lemmas which will allow us to insert well-behaved paths into a disk diagram.}
	
\vspace{.4cm}	

\begin{definition}[Disk diagram surgery] \label{def:surgery}
	Let $D$ be a disk diagram, and let $\gamma$ be a simple path in $D$ with label $w$. Let $w'$ be a reduced expression for $w$. 
	Let $D'$ be a disk diagram with boundary path $\gamma' \eta^{-1}$ such that the labels of $\gamma'$ and $\eta$ are $w'$ and $w$ respectively. 
	Let $D''$ be the disk diagram consisting of a copy of $D'$ and another reflected copy of $D'$ glued together along the path $\gamma'$
	(see Figure \ref{fig:surgery}). Note that the boundary path of $D''$ has label $w w^{-1}$.
	
	We first slightly thicken $\gamma$ in $D$ and then cut along this path to produce an annular diagram $A$, one of whose boundary paths has label $ww^{-1}$ and the other has label the same as that of a boundary path of $D$. 
	We then attach $D''$ along its boundary to the boundary of $A$ with label $ww^{-1}$.
	Let $E$ be this resulting disk diagram.
	We can naturally think of $\gamma'$ as a path in $E$ with label $w'$.
	We say that the resulting diagram is \textit{obtained from $D$ by surgery to insert $\gamma'$ in place of $\gamma$} and that \textit{$E$ is obtained from $D$ by surgery}. 
	Note that $E$ contains two (possibly equal) paths, labeled by $w$, which naturally correspond to $\gamma$ along the boundary of the inserted disk $D''$. We say that these paths are \textit{copies of $\gamma$} in $E$. By a slight abuse of notation, we will often refer to $\gamma$ when we mean a copy of $\gamma$.
\end{definition}

\begin{figure}[htbp]
	\begin{overpic}[scale=.3]{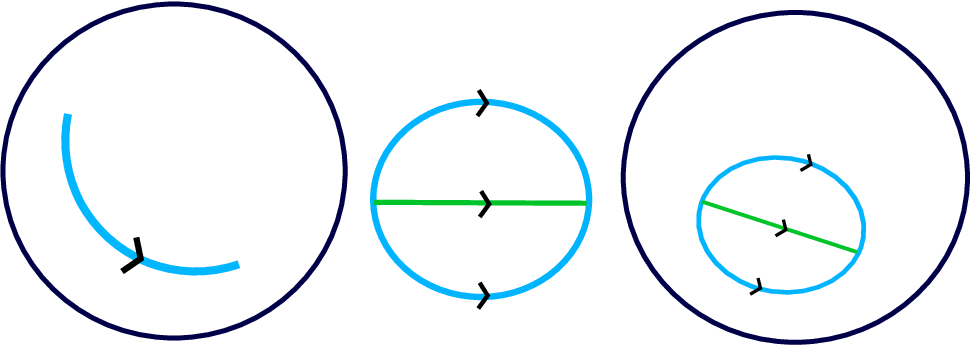}
		\put(11,12){$w$}	
		
		\put(47,27){$w$}
		\put(47,16){$w'$}
		\put(47,0){$w$}
		
		\put(13,-8){$D$}
		\put(47,-8){$D''$}
		\put(80,-8){$E$}
	\end{overpic}
	\caption{Disk diagram surgery.}
	\label{fig:surgery}
\end{figure}

Fix the notation from the previous definition.
There is a natural map $\Psi: E \to D$ collapsing the disk which was inserted into $D$.
We need to take great care when performing surgeries.
For instance, given a dual curve $H$ in $E$, it could be that $\Psi(H)$ is contained in two distinct dual curves of $D$ (see Figure~\ref{fig:nongon}). 
Moreover, 
bigons and nongons (as described in \cite{Wise-riches}) can be introduced after surgery, even if they are not present in $D$ or $D''$. For instance, see Figures~\ref{fig:nongon} and \ref{fig:tracking_lemma} for an example where surgery creates respectively a nongon and a bigon.

\begin{figure}[htbp]
	\begin{overpic}[scale=.3]{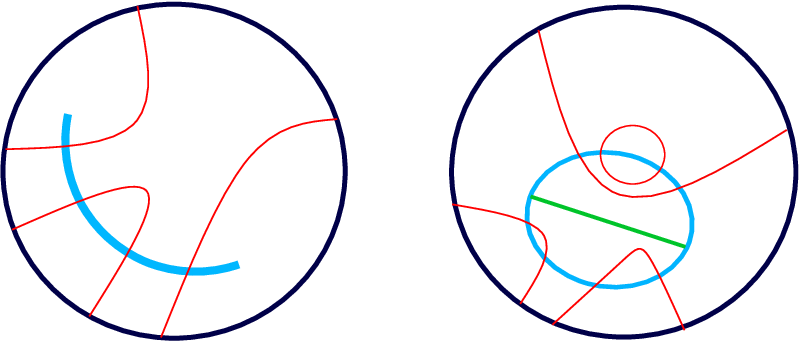}
	\end{overpic}
	\caption{The diagram on the right is obtained by surgery from the diagram on the left. The red curves on the right map to the ones on the left via $\Psi$.}
	\label{fig:nongon}
\end{figure}
However, some things can be seen to be preserved by surgery.
For instance, the boundary path of $E$ is canonically identified with a boundary path of $D$ (even if $\gamma$ contains edges of~$\partial D$).
Moreover, 
given a subdiagram $G \subset D$ whose interior does not intersect $\gamma$, it follows that $G$ is naturally a subdiagram of $E$.  

The next convention will from now on be used as a book-keeping device for paths which track other paths (as in the definition below).
\begin{convention} \label{conv:hat}
As a convention, we will always use hat notation as follows: the word $\hat{w}$ will always be understood be equal to a word $w$ with some letters deleted. Similarly, the path $\hat{\gamma}$ will always be understood to track the path $\gamma$ (as defined below).
\end{convention}

\begin{definition}[Track]
	Let $\gamma$ and $\gamma'$ be oriented paths in a disk diagram. 
	We say that $\gamma'$ \textit{tracks} $\gamma$ if $\gamma'$ is reduced and the following holds.
	Let $H_1, \dots, H_n$ be the set of all dual curves which intersect $\gamma'$ ordered by the orientation of $\gamma'$, i.e. $H_i \cap \gamma'$ occurs prior to $H_{i+1} \cap \gamma'$ along the orientation of $\gamma'$ for all $1 \le i < n$. 
	Then $H_1, \dots, H_n$ each intersect $\gamma$ and are ordered along the orientation of $\gamma$. In particular, if $w$ is the label of $\gamma$, then the label of $\gamma'$ is a reduced word $\hat{w}$ (as in Convention~\ref{conv:hat}). 
\end{definition}

	By definition, if $\hat{\gamma}$ tracks $\gamma$ and $\hat{\hat{\gamma}}$ tracks $\hat{\gamma}$, then $\hat{\hat{\gamma}}$ tracks $\gamma$.
	Additionally, if $\hat{\gamma}$ tracks $\gamma$, then a sequence of dual curves (resp. spokes) structured with respect to $\hat{\gamma}$ is also structured with respect to $\gamma$.
	These observations will be frequently used without mention.

The next two lemmas guarantee we can use disk diagram surgery to insert paths with certain desirable properties into a disk diagram.

\begin{lemma} \label{lem:tracking paths}
	Let $\gamma$ be a simple path with label $w$ in a disk diagram $D$. We can obtain a disk diagram $E$ by applying surgery to $D$ to insert a path $\hat{\gamma}$ in place of $\gamma$, such that $\hat{\gamma}$ tracks each copy of $\gamma$ in $E$ and has label a reduced expression $\hat{w}$ for $w$.
\end{lemma}
\begin{proof}
 	Let $D'$ be a disk diagram with boundary path $\eta (\gamma')^{-1}$ where $\eta$ and $\gamma'$ have labels $w$ and  $w'$ respectively, where $w'$ is a reduced expression for $w$. Additionally choose $D'$ to have minimal area out of all such possible diagrams.  
 	By \cite{Wise-hierarchy}[Lemma 2.6 and Corollary 2.7], $D'$ has the property that no two dual curves emanating from distinct edges of $\gamma'$ intersect. As no dual curve intersects a reduced path twice, every dual curve dual to $\gamma'$ intersects $\eta$ in $D'$. Consequently, $\gamma'$ tracks $\eta$ in $D'$. We apply surgery to $D$ as in Definition~\ref{def:surgery} to obtain the disk diagram $E$ by, as in this definition, inserting a disk diagram $D''$ which consists of two glued copies of $D'$. It follows that $\gamma'$ tracks the copies of $\gamma$ in $E$. Thus, we can set $\hat{\gamma} = \gamma'$ and $\hat{w} = w'$.
\end{proof}

\begin{figure}[htbp]
	\begin{overpic}[scale=.4]{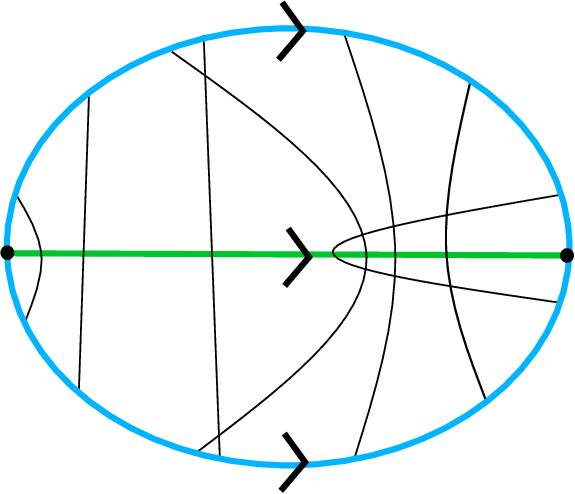}
		\put(20,78){\tiny$\gamma$}
		\put(20,7){\tiny$\gamma$}
		\put(20,35){\tiny$\gamma'$}
		\put(4,50){\tiny$H_1$}	
		\put(15,60){\tiny$H_2$}
		\put(36,50){\tiny$H_3$}
		\put(50,62){\tiny$H_5$}	
		\put(65,67){\tiny$H_6$}
		\put(80,60){\tiny$H_7$}
		\put(85,52.5){\tiny$H_4$}
	\end{overpic}
	\caption{Example of the subdiagram with boundary path $\gamma \gamma^{-1}$ of the disk diagram in the conclusion of Lemma~\ref{lem:tracking_paths2}. The black lines represent \textit{all} hyperplanes of this subdiagram. In this example $j = 4$.  }
		\label{fig:tracking_lemma}
\end{figure}

\begin{lemma} \label{lem:tracking_paths2}
	Let $D$ be a disk diagram over the RACG $W_\Gamma$.
	Let $\gamma$ be a path in $D$ with reduced label $w = s_1 \dots s_n$. 
	Then we can apply surgery to $D$ to insert a path $\gamma'$ in place of $\gamma$ with label $w' = s_1' \dots s_n'$ a reduced expression for $w$. Moreover, there exists a $1 \le j \le n$ such that $s_j' = s_n$ and the resulting diagram satisfies:
	\begin{enumerate}
		\item For all $j < i \le n$, $s_j'$ and $s_i'$ are adjacent vertices of $\Gamma$.
		\item The initial subpath of $\gamma'$ with label $s_1' \dots s_{j}'$ tracks each copy of $\gamma$.
		\item For each $1 \le i < j$, the dual curve dual to the edge of $\gamma'$ labeled by $s_i'$ does not intersect the dual curve dual to the edge of $\gamma'$ labeled by $s_j'$.
	\end{enumerate}
\end{lemma}
\begin{proof}
	Let $w' = s_1' \dots s_n'$ be an expression for $w$ and $1 \le j \le n$ be such that $s_j' = s_n$ and, for all $j < i \le n$, $s_j'$ and $s_i'$ are adjacent vertices of $\Gamma$. 
	We additionally choose $w'$ and $j$ so that $j$ is minimal out of all such possible choices.
	Note that such an expression exists as we can take $w' = w$ and $j = n$ (where $j$ is not necessarily minimal).

	Set $h_1 = s_1' \dots s_{j-1}'$ and $h_2 = s_{j+1}' \dots s_n'$ and note that $w' = h_1 s_j' h_2$. 
	By \cite{Wise-hierarchy}[Lemma~2.6 and Corollary~2.7], it follows that, by 
	possibly replacing $h_1$ in $w'$ with another expression for $h_1$, 
	there is a disk diagram $D'$ with boundary path labeled $w w'^{-1}$
	such that there are no intersection between pairs of dual curves dual to the subpath of the boundary path of $D'$ corresponding to $h_1$.
	In particular, as dual curves intersect a reduced path at most once, the subpath of the boundary path of $D'$ labeled by $s_1' \dots s_{j}'$ tracks the subpath labeled by $w$.
	
	We now apply surgery to $D$ to insert two glued copies of $D'$ into $D$ as in Definition \ref{def:surgery}. 
	Let $E$ be the resulting disk diagram, and let
	 $\gamma'$ be the corresponding path labeled by $w'$ in $E$.
	We claim that, for each $1 \le k < j$,  the dual curve dual to the edge of $\gamma'$ labeled by $s_k'$ does not intersect the dual curve dual to the edge labeled by $s_j'$. 
	For suppose, otherwise, and take $k$ maximal with this property.
	It then follows by the commuting relations imposed by intersections of dual curves in $E$ that $s_1' \dots s_{k-1}' s_{k+1}' \dots s_j' s_k' s_{j+1}' \dots s_n'$ is an expression for $w$ and that $s_k'$ is adjacent to $s_j'$ in $\Gamma$, contradicting our choice of $j$ being minimal. 
	The lemma now follows.
\end{proof}

\section{Divergence Bounds} \label{sec:div_bounds}

\emph{The main result of this section is Theorem~\ref{thm:main} which gives a lower bound on the length of certain paths avoiding an $R$--ball in the Davis complex. We wish to prove this theorem by induction; however, for this argument to work, we first need to  prove the technical Propositions \ref{prop:main_full} and \ref{prop:main_pieces}.}

\begin{figure}[htp]
	\centering
	\begin{overpic}[scale=.4]{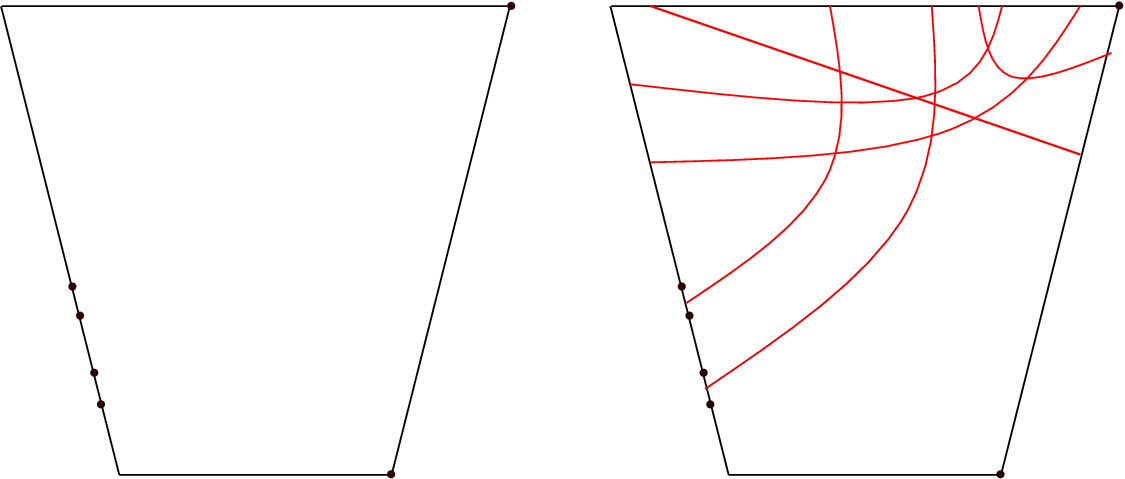}
		\put(7,7){\Tiny $e$}
		\put(5,14.5){\Tiny$f$}
		\put(40,18){\Tiny$\eta$}
		\put(20,-4){\Tiny $D$}
		
		\put(61,7){\Tiny $e$}
		\put(59,14.5){\Tiny$f$}
		\put(95,18){\Tiny$\eta$}
		
		\put(67,14){\Tiny$H_1$}
		\put(66,22){\Tiny$K_1$}
		
		\put(65,29){\Tiny$H_2$}
		\put(62,35){\Tiny$K_2$}
		
		\put(90,28){\Tiny$H_3$}
		\put(93,34){\Tiny$K_3$}
	\end{overpic}
	\caption{Suppose that the diagram on the right is obtained from $D$ via a sequence of surgeries and that $\{\{H_1, K_1\}, \{H_2,K_2\}, \{H_3,K_3\}\}$ is a $1$--fence. Then the edge-pair $\{e,f\}$ is \emph{not} $1$--fence separated from $\eta$ with respect to $D$.}
	\label{fig:separated}
\end{figure}
We establish some terminology.
Given a disk diagram $D$ with boundary path $\gamma$ and a subpath $\alpha \subset \gamma$, there is a closed path $\gamma'$ in the $1$--skeleton of the Davis complex $\Sigma_\Gamma$ (i.e., the Cayley graph of $W_{\Gamma}$ with standard generators) based at the vertex representing the identity and with the same label as that of $\gamma$. Furthermore, there is a subpath $\alpha' \subset \gamma'$ naturally corresponding to $\alpha$. We say that $\alpha$ is \textit{$R$--avoidant with respect to $D$} if the corresponding path $\alpha'$ in $\Sigma_\Gamma$ does not intersect the ball of radius $R$ about the vertex representing the identity element.

We would like to define a spoke to be ``$L$--fence separated'' from a path, if it is not contained in any $L$--fence that intersects this path. However, we need this property to also be preserved under disk diagram surgeries. This motivates the following definition (see also Figure~\ref{fig:separated}).

Let $D$ be a disk diagram with boundary path $\gamma$, let $\eta$ be a subpath of $\gamma$, and let $\{e_1, e_2\}$ be an edge-pair in $\gamma$.
We say that $\{e_1, e_2\}$ is \textit{$L$--fence separated from $\eta$ with respect to $D$} if given any disk diagram $E$  obtained from $D$ by a series of surgeries and a spoke $\{H, K\}$ in $E$ dual to $\{e_1, e_2\}$, then $\{H,K\}$ is not contained in an $L$--fence which intersects $\eta$.

We are now ready to state the main theorem of this section:

\begin{theorem} \label{thm:main}
	Let $D$ be a disk diagram with boundary path $\gamma \alpha \eta \beta$ over the RACG $W_\Gamma$ such that:
	\begin{enumerate}
		\item $\gamma$ is reduced,
		\item every edge-pair in $\gamma$ is $L$--fence separated from $\eta$, and
		\item $\alpha$ is $R$--avoidant with respect to $D$ where $R$ is larger than some fixed universal constant.
	\end{enumerate}
	Then $|\alpha| \ge C R^{L+1}$ where the constant $C$ depends only on $L$ and $\Gamma$.
\end{theorem}

The two technical propositions below are required to prove the above theorem. In these propositions, we utilize the function $f_L^M(R) := \frac{R}{M^{100L + 50}}$ which we denote by $f_L(R)$ when $M$ is implicit. Before turning to their proofs, we first show that Proposition~\ref{prop:main_full} implies Theorem~\ref{thm:main}:
\begin{proof}[Proof of Theorem~\ref{thm:main}:]
	Let $D$ be a diagram with boundary path $\gamma \alpha \eta \beta$ as in the statement of Theorem~\ref{thm:main}. Define $\gamma_1$ to be the starting point of $\gamma$ and $\gamma_3$ to be the endpoint of $\gamma$ (i.e., $\gamma_1$ and $\gamma_3$ are length $0$ paths). Clearly the hypotheses of Proposition~\ref{prop:main_full} hold with these choices. The theorem now follows from conclusion (B4) of Proposition~\ref{prop:main_full}.
\end{proof}

\begin{figure}[h]
	\centering
	\begin{subfigure}{.49\textwidth}
		\centering
	\begin{overpic}[scale=.41]{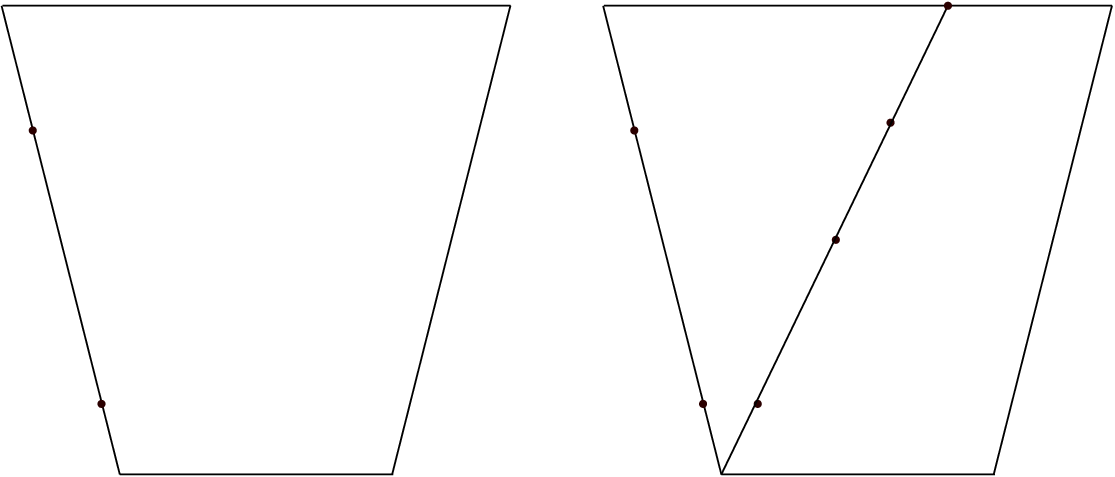}
		\put(20,-5){\Tiny $D$}	
		\put(6,2){\Tiny $\gamma_1$}	
		\put(2,19){\Tiny $\gamma_2$}	
		\put(-2,35){\Tiny $\gamma_3$}	
		\put(20,43){\Tiny $\alpha$}
		\put(40,19){\Tiny $\eta$}
		\put(21,1){\Tiny $\beta$}
		
		\put(48,15){\Tiny $\to$}
		
		\put(75,-5){\Tiny $E$}		
		\put(65,25){\Tiny $E'$}	
		\put(60,2){\Tiny $\gamma_1$}	
		\put(56,19){\Tiny $\gamma_2$}	
		\put(52,35){\Tiny $\gamma_3$}	
		\put(66,2){\Tiny $\gamma_1'$}	
		\put(71,12){\Tiny $\gamma_2'$}	
		\put(77.5,25){\Tiny $\zeta$}	
		\put(82,35){\Tiny $\pi$}		
		\put(68,43){\Tiny $\alpha'$}
		\put(94,19){\Tiny $\eta$}
		\put(75,1){\Tiny $\beta$}

	\end{overpic}
	\vspace*{5mm} 
	\caption{Disk diagrams of Proposition~\ref{prop:main_full}.}
	\end{subfigure}
	\begin{subfigure}{.49\textwidth}
		\centering
	\begin{overpic}[scale=.41]{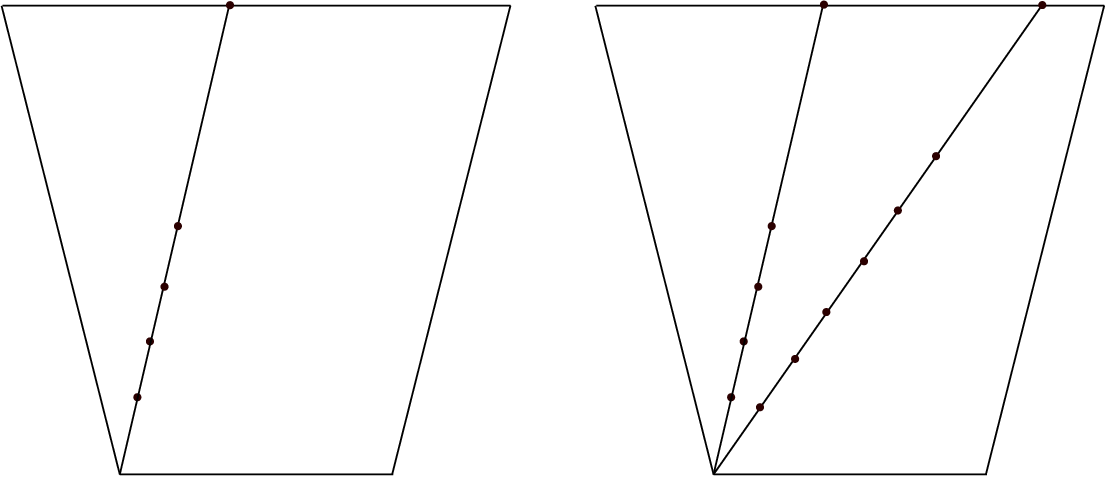}
		\put(8,33){\Tiny $D_1$}	
		\put(30,33){\Tiny $D_2$}	
		\put(20,-5){\Tiny $D$}	
		\put(2,22){\Tiny $\gamma$}	
		\put(16,31){\Tiny $\theta$}
		\put(11.5,20){\Tiny $\mu_2$}		
		\put(10.5,15){\Tiny $\mu_1$}		
		\put(9.5,10){\Tiny $\mu_0$}		
		\put(10,5){\Tiny $\psi$}
		\put(10,43){\tiny $\kappa_1$}
		\put(30,43){\tiny $\kappa_2$}
		\put(41,19){\Tiny $\eta$}
		\put(21,1){\Tiny $\beta$}

		\put(48,15){\Tiny $\to$}

		\put(62,33){\Tiny $D_1$}	
		\put(75,-5){\Tiny $E$}	
		\put(56,22){\Tiny $\gamma$}	
		\put(70,31){\Tiny $\theta$}
		\put(65,20){\Tiny $\mu_2$}		
		\put(64.5,15){\Tiny $\mu_1$}		
		\put(63.2,10){\Tiny $\mu_0$}		
		\put(63.5,5){\Tiny $\psi$}			

		\put(89,33){\Tiny $\theta'$}
		\put(83,25.5){\Tiny $\mu_4'$}		
		\put(80,21){\Tiny $\mu_3'$}		
		\put(76.5,16){\Tiny $\mu_2'$}		
		\put(73.5,12){\Tiny $\mu_1'$}		
		\put(70,7){\Tiny $\mu_0'$}		
		\put(67,2){\Tiny $\psi'$}	
		\put(94,19){\Tiny $\eta$}
		\put(75,1){\Tiny $\beta$}
		\put(82,43){\tiny $\alpha'$}
	\end{overpic}
	\vspace*{5mm} 
	\caption{Disk diagrams of Proposition~\ref{prop:main_pieces} for $n=2$.}
	\end{subfigure}
	\caption{Disk diagrams in the statements of Propositions~\ref{prop:main_full} and \ref{prop:main_pieces}.}
\end{figure}

\begin{proposition} \label{prop:main_full}
	Let $D$ be a disk diagram with simple boundary path $\gamma_1 \gamma_2 \gamma_3 \alpha \eta \beta$ and basepoint $b$. Let $M > \max \{|V(\Gamma)| + 1, |\beta| \}$, $L 
	\ge 0$ and $R >0$ be integers. Furthermore, suppose that:
	
	\begin{enumerate}
		\item[(A1)]
			The path $\alpha$ is $R$--avoidant with respect to $D$.
		
		\item[(A2)]
			The path $\gamma_1\gamma_2$ is reduced, no dual curve is dual to both $\gamma_1 \gamma_2$ and $\gamma_3$,
			 $|\gamma_1| \le f_L(R)$ and $|\gamma_2| \ge f_L(R)$.
		
		\item[(A3)]  
			There is an $M$--adequate set $\mathcal{R}(\gamma_2)$ of edge-pairs which are $L$--fence separated from $\eta$ with respect to $D$.
			
	\end{enumerate}
	Then, for $R$ large enough (depending only on $M$ and $L$), we can apply a sequence of surgeries to $D$ to obtain a disk diagram $E$ containing a path $\gamma_1' \gamma_2' \zeta \pi$ from $b$ to $\alpha$. Let $E' \subset E$ be the subdiagram with boundary path
	 $\gamma_1 \gamma_2 \gamma_3 \alpha' (\gamma_1' \gamma_2' \zeta \pi)^{-1}$ where $\alpha'$ is the subpath of $\alpha$  between the endpoint of $\gamma_3$ and the endpoint of $\pi$. We additionally have that:
	\begin{enumerate}
		\item[(B1)]
			The path $\gamma_1' \gamma_2' \zeta$ is reduced, no dual curve is dual to both $\gamma_1'\gamma_2'\zeta$ 
			 and $\pi$,  $|\gamma_2' \zeta| \ge 32f_{L+1}(R)$,  $\gamma_1'$ tracks $\gamma_1$ and $\gamma_2'$ tracks $\gamma_2$.
		\item[(B2)]
			Let $\mathcal{R}(\zeta)$ be the set of all edge-pairs $\{e_1, e_2\}$ in $\zeta$ satisfying the following:  
			 there exists an $(L+1)$--fence in $E'$ containing a spoke dual to $\{e_1, e_2\}$ and a spoke of  $\mathcal{R}(\gamma_2)$.
			Then $\mathcal{R}(\zeta)$ is $M$--adequate.
		\item[(B3)]
			Every dual curve which intersects $\zeta$ either intersects $\gamma_2$ or intersects an $L$--fence in $E'$ that contains a spoke in $\mathcal R (\gamma_2)$.
		
		\item[(B4)] 
			The path $\alpha'$ has length at least $C_L R^{L+1}$ where $C_L$ depends only on $L$ and $M$.
	\end{enumerate}
\end{proposition}

\begin{proposition} \label{prop:main_pieces}
	Let $D$ be a disk diagram with simple boundary path $\gamma \alpha \eta \beta$ and basepoint $b$. 
	Let $M > \max \{|V(\Gamma)| + 1, |\beta| \}$, $L 
	\ge 1$, $n \le \frac{f_L(R)}{50M^4} - 1$ and $R > 0$ be integers.
	Let $\psi \mu_0 \dots \mu_n \theta$ be a simple path
	 in $D$ from $b$ to $\alpha$.
	Let $D_1 \subset D$ and $D_2 \subset D$ be the subdiagrams with boundary paths $\gamma \kappa_1 (\psi \mu_0 \dots \mu_n \theta)^{-1}$ and $\psi \mu_0 \dots \mu_n \theta \kappa_2 \eta \beta$ respectively, where $\alpha = \kappa_1 \kappa_2$ and the endpoint of $\kappa_1$ is the endpoint of $\theta$.
	Suppose that in $D_1$ there is an $L$--fence $\mathcal{F}$ that contains a spoke intersecting an edge-pair that is $L$--fence separated from $\eta$ with respect to $D$.
	Additionally, suppose~that:
	\begin{enumerate}
		\item[(X1)] 
			The path $\alpha$ is $R$--avoidant with respect to $D$.
		\item[(X2)]
			The path $\psi \mu_0  \dots \mu_n$ is reduced, no dual curve intersects both $\psi \mu_0 \dots \mu_n$ and $\theta$, $|\psi| \le f_L(R)$ and $|\mu_0 \dots \mu_n| \ge f_L(R)$.
		\item[(X3)]
			For each $1 \le i \le n$, let $\mathcal{R}(\mu_i)$ be the set of all edge-pairs in $\mu_i$ 
			that are dual to a 
			spoke contained in $\mathcal{F}$. Then  $\mathcal{R}(\mu_i)$ is $M$--adequate.
		\item[(X4)]
			The set of edges in $\mu_1 \dots \mu_n$ which are dual to a dual curve that intersects $\mathcal{F}$ is an $M$--adequate edge set.
		\item[(X5)] 
			There is an $M$--adequate set $\mathcal{R}(\mu_0)$ of edge-pairs in $\mu_0$ which are $L$--fence separated from $\eta$ with respect to $D_2$.
	\end{enumerate}
	Then, for $R$ large enough (depending only on $M$ and $L$), we can apply a sequence of surgeries to $D$, none of which involves a path that intersects the interior of $D_1 \subset D$, to obtain a disk diagram $E$. Moreover, $E$ contains a simple path $\psi' \mu_0' \dots \mu_{n+2}' \theta'$ from $b$ to $\kappa_2$ that does not intersect the interior of the image of $D_1$ in $E$. Additionally, we have that:
	\begin{enumerate}
		\item[(Y1)]
			Let $E' \subset E$ be the subdiagram with boundary path $\gamma \kappa_3 (\psi' \mu_0' \dots \mu_{n+2}' \theta')^{-1}$ where $\kappa_3$ is the initial subpath of $\alpha$ up to the endpoint of $\theta'$.
			There exists an $L$--fence $\mathcal{F}'$ in $E'$ which either contains every spoke of $\mathcal{F}$, or alternatively contains a spoke in~$\mathcal{R}(\mu_0)$.
		\item[(Y2)]
			The path $\psi' \mu_0' \dots \mu_{n+2}'$ is reduced, no dual curve intersects both $\psi' \mu_0' \dots \mu_{n+2}'$ and $\theta'$, and the paths $\psi', \mu_0', \dots, \mu_n'$ each track $\psi, \mu_0, \dots, \mu_n$ respectively. 					
		\item[(Y3)]
			For all $1 \le i \le n+2$, the set of edge-pairs in $\mu_i'$ 
			that are dual to a spoke in $\mathcal{F}'$ is $M$--adequate.
		\item[(Y4)]
			The set of edges in $\mu_1' \dots \mu_{n+2}'$ that are dual to a dual curve that intersects $\mathcal F '$ is an $M$--adequate edge set.
		\item[(Y5)]
			In $E'$, a dual curve which intersects $\mu_{n+1}'\mu_{n+2}'$ also intersects either $\mathcal{F}'$ or $\mu_0\mu_1 \dots \mu_n$.
		\item[(Y6)] 
			Let $\alpha'$ be the subpath of $\alpha$ from the endpoint of $\theta$ to the endpoint of $\theta'$. One of the two possibilities holds:
			\begin{enumerate}
				\item 
					 $|\mu_0' \dots \mu_{n+2}'| \ge 32 f_{L+1}(R)$ and $|\alpha'| \ge C_L R^{L+1}$ where $C_L$ depends only on $M$ and~$L$.
					 \
				\item 
					$|\mu_0' \dots \mu_{n+2}'| \ge f_{L}(R)$ and $|\alpha'| \ge C_L' R^{L}$ where $C_L'$ depends only on $M$ and~$L$. 
			\end{enumerate}
	\end{enumerate}
\end{proposition}

\begin{remark}
We note that the constant $32$ in (Y6) and (B1) is chosen so that the equality $32f_l(r/16) = 2f_l(r)$ holds (which will be utilized in the proof). Moreover, $f_l(r)$ is chosen to decrease exponentially with respect to $l$ and so that $f_{l+1}(r)$ is much smaller than $32f_{l}(r)$ for all integers $l \ge 0$.
\end{remark}

Our strategy in proving the above two propositions is the following. We first show that Proposition \ref{prop:main_full} holds when $L=0$. Next, we show that
 Proposition \ref{prop:main_pieces} holds for $L = \ell \ge 1$ given that
 Proposition \ref{prop:main_full} holds for $L = \ell-1$. 
 Finally, we show that if Proposition \ref{prop:main_pieces} holds for $L = \ell \ge 1$ then Proposition \ref{prop:main_full} also holds for $L = \ell$.

\begin{figure}[h]
	\centering
		\begin{overpic}[scale=.41]{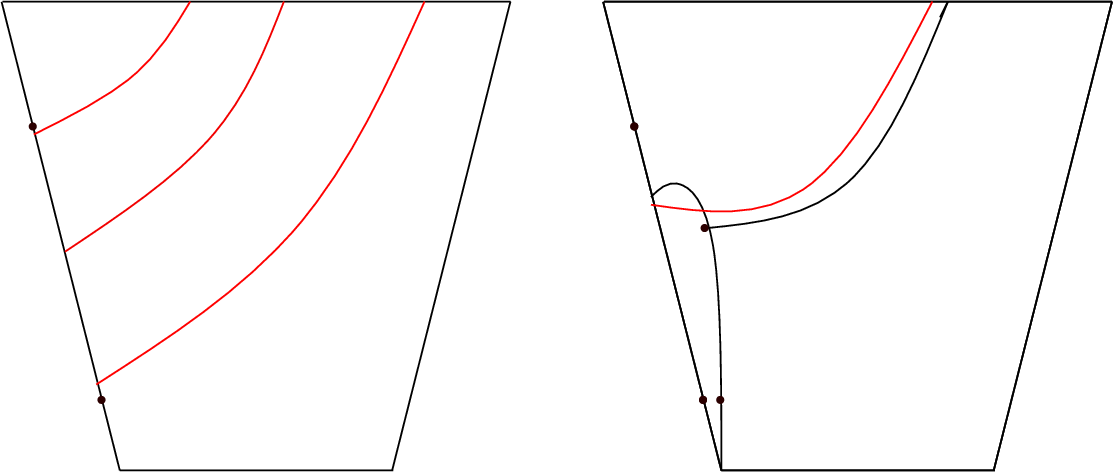}
			\put(20,-4){\Tiny $D$}	
			\put(6,2){\Tiny $\gamma_1$}	
			\put(2,15){\Tiny $\gamma_2$}	
			\put(-2,35){\Tiny $\gamma_3$}	
			\put(20,43){\Tiny $\alpha$}
			\put(40,19){\Tiny $\eta$}
			\put(21,1){\Tiny $\beta$}
			\put(21,15){\Tiny $H_1$}
			\put(12,22.5){\Tiny $H_{33c}$}
			\put(10,32){\Tiny $H_{t}$}
			
			\put(48,15){\Tiny $\to$}
			
			\put(60,2){\Tiny $\gamma_1$}	
			\put(56,19){\Tiny $\gamma_2$}	
			\put(52,35){\Tiny $\gamma_3$}	
			\put(65,2){\Tiny $\hat{\gamma}_1$}	
			\put(65,12){\Tiny $\hat{\gamma}_2$}	
			\put(82,35){\Tiny $\pi$}	
			\put(70,32){\Tiny $H_{33c}$}			
			\put(94,19){\Tiny $\eta$}
			\put(75,1){\Tiny $\beta$}
			
		\end{overpic}
		\caption{Disk diagrams for the proof of Proposition~\ref{prop:main_full} when $L=0$.}
	\end{figure}

\begin{proof}[Proof of Proposition \ref{prop:main_full} for $L=0$.]	
	Let $D$ be a disk diagram as in Proposition \ref{prop:main_full} with $L = 0$.
	Set $c := f_1(R)$.
	By Lemma \ref{lem:enough_curves} (where in that lemma we set $k=0$, $\psi= \gamma_1$, $\mu_0 = \gamma_2$ and $\theta = \gamma_3$)
	there exists a sequence of spokes $\{H_1, K_1\}, \dots, \{H_t, K_t\}$ in $\mathcal{R}(\gamma_2)$ and structured with respect to $(\gamma_2, \alpha)$ such that 
	\[ t \ge \frac{1}{M^2} \Big \lfloor \frac{|\gamma_2|}{M} \Big \rfloor - 3M \ge \frac{1}{M^2} \Big \lfloor \frac{f_0(R)}{M} \Big \rfloor - 3M \ge \frac{f_0(R)}{2M^3} \ge 34 f_1(R) = 34c \]
	where the second and third inequality follow respectively from (A2) and $R$ being large enough.
	
	Let $e$ be the edge of $\gamma_2$ dual to $H_{33c}$. 
	Let $\rho$ be the initial subpath of $\gamma_1\gamma_2$ up to and including $e$, and
	let $w = s_1 \dots s_n$ be its label. 
	As $\{H_1, K_1\}, \dots, \{H_t, K_t\}$ is a structured sequence, there exist non-adjacent vertices $s$ and $t$ of $\Gamma$ such that $H_i$ and $K_i$ are of type $s$ and $t$ respectively for all $1 \le i \le t$.
	In particular, $s_n = s$.
	
	Using Lemma \ref{lem:tracking_paths2}, we apply a disk diagram surgery to replace $\rho$ with a reduced path $\rho'$, with label $w' = s_1' \dots s_n'$ 
	such that, for some $1 \le j \le  n$: (1) $s_j'$ is of type $s$, (2) $s_i'$ and $s_j'$ are adjacent vertices of $\Gamma$ for all $i > j$ and (3) for every $i < j$, the dual curve dual to the edge of $\rho'$ labeled by $s_i'$ does not intersect the dual curve dual to the edge of $\rho'$ labeled by $s_j'$. 	
	
	As both $\rho$ and $\rho'$ are reduced and by (1) and (2) above, it follows that $H_{33c}$ (in this resulting diagram) intersects the edge $e' \subset \rho'$ labeled by $s_j'$.
	Let $\rho''$ be the initial subpath of $\rho'$ 
	up to, and not including, $e'$.
	We have that $\rho'' = \hat{\gamma}_1 \hat{\gamma}_2$ (with the notation as in Convention \ref{conv:hat}) where $ \hat{\gamma}_1$ (resp.  $\hat{\gamma}_2$) tracks $\gamma_1$ (resp. $\gamma_2$).
	Furthermore, as structured dual curves are pairwise non-intersecting, for all $1 \le i < 33c$, $H_i$ intersects $\hat{\gamma}_2$ and in particular $|\hat{\gamma}_2| > 32c$. 
	
	We define $\gamma_1' = \hat{\gamma}_1$ and $\gamma_2'=\hat{\gamma}_2$.
	We set $\zeta$ to be the endpoint of $\gamma_2'$ (i.e., a length $0$ path).
	Finally, we set $\pi$ to be a simple path in the carrier $N(H_{33c})$ from the endpoint of $\gamma_2'$ to $\alpha$.

	We now check that the conclusions of Proposition \ref{prop:main_full} are satisfied with these choices of $\gamma_1'$, $\gamma_2'$, $\zeta$ and $\pi$.
	By our application of Lemma \ref{lem:tracking_paths2}, $\gamma_1'\gamma_2'\zeta$ is reduced (as $\hat{\gamma}_1 \hat{\gamma}_2$ is reduced), $\gamma_1'$ tracks $\gamma_1$, $\gamma_2'$ tracks $\gamma_2$, any dual curve dual to $\gamma_1'\gamma_2'$ does not intersect $\pi$ (as it does not intersect $H_{33c}$ by (3) above) and $|\gamma_2' \zeta| \ge 32c = 32f_1(R)$. Thus, (B1) follows.

	Conclusions (B2) and (B3) hold trivially as $|\zeta| = 0$.
	To see (B4), note that since the dual curves $H_1, K_1, \dots, H_t, K_t$ are pairwise non-intersecting, it follows that, for $33c \le i \le 34c$, $H_i$ intersects the subpath of $\alpha$ between the endpoint of $\gamma_3$ and the endpoint of $\pi$.
	In particular, this subpath has length at least $c = f_1(R) \ge C_0 R$ where $C_0$ depends only on $M$. 
	Thus, the conclusions of Proposition~\ref{prop:main_full} hold.
\end{proof}

\begin{figure}[h]
	\centering
	\begin{subfigure}{.49\textwidth}
		\centering
		\begin{overpic}[scale=.5]{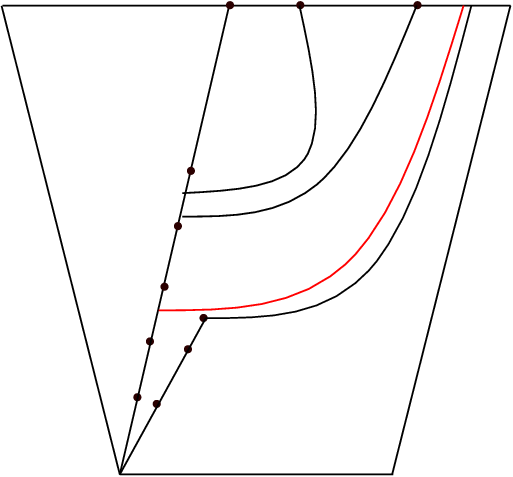}
			\put(50,40){\Tiny $H_{33c}$}	
			\put(70,80){\Tiny $h_i$}	
			\put(55,80){\Tiny $k_i$}	
			\put(63,70){\Tiny $D_i$}	
			\put(58,95){\Tiny $|\alpha_i| \approx R^L$}	
			\put(22,9){\Tiny $\psi$}
			\put(21,20){\Tiny $\mu_0$}
			\put(23,30){\Tiny $\mu_1$}
			\put(26,42){\Tiny $\mu_2$}
			\put(29,54){\Tiny $\mu_3$}
			\put(36,70){\Tiny $\theta$}
			
			\put(29,7){\Tiny $\psi'$}
			\put(33,17){\Tiny $\mu_0'$}
			\put(38.5,26){\Tiny $\mu_1'$}
			\put(65,30){\Tiny $\theta'$}

		\end{overpic}
		\caption{Case 1 where $m = 1$, $n=3$ and $i > 33c$.}
	\end{subfigure}
	\begin{subfigure}{.49\textwidth}
		\centering
		\begin{overpic}[scale=.5]{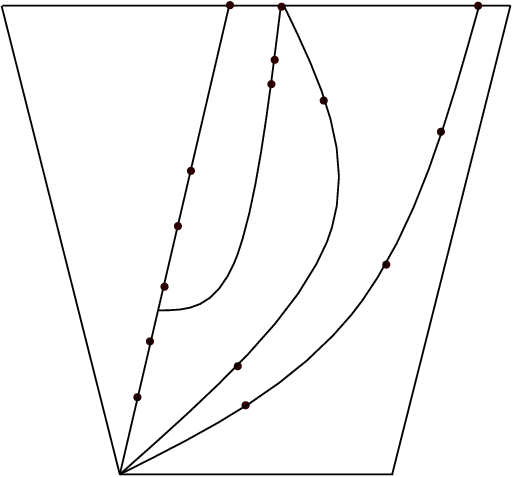}
		\put(22,9){\Tiny $\psi$}
		\put(21,20){\Tiny $\mu_0$}
		\put(23,30){\Tiny $\mu_1$}
		\put(26,42){\Tiny $\mu_2$}
		\put(29,54){\Tiny $\mu_3$}
		\put(36,70){\Tiny $\theta$}
		
		\put(42,50){\Tiny $h_j'$}
		\put(49,78){\Tiny $a$}
		\put(47,85){\Tiny $h_j''$}
	
		\put(35,17){\Tiny $\gamma_1$}
		\put(66,60){\Tiny $\gamma_2$}
		\put(61,80){\Tiny $\gamma_3$}
		
		\put(38,5){\Tiny $\gamma_1'$}
		\put(60,20){\Tiny $\gamma_2'$}
		\put(83,55){\Tiny $\zeta$}
		\put(91,80){\Tiny $\pi$}

		\put(60,95){\Tiny $|\alpha''| \approx R^L$}

		\end{overpic}
		\caption{Case 2.}
	\end{subfigure}
	\caption{Proof of Proposition \ref{prop:main_pieces} assuming Proposition \ref{prop:main_full} for $L = \ell-1$.}
\end{figure}

\begin{proof}[Proof of Proposition \ref{prop:main_pieces} for $L = \ell \ge 1$, assuming Proposition \ref{prop:main_full} for $L = \ell-1$]

	Suppose that $D$ is a disk diagram satisfying the hypotheses of Proposition \ref{prop:main_pieces} with $L = \ell$ and $R = r$ and that the statement of Proposition \ref{prop:main_full} is true for $L = \ell - 1$. 

	Set $c := f_{\ell+1}(r)$. 
	Let $\nu$ be the initial subpath of $\mu_0 \dots \mu_n$ of length $\lceil f_{\ell}(r) \rceil$, which exists by (X2). 
	By Lemma \ref{lem:enough_curves}, (X2), (X3) and (X5) there exists a tight sequence of  spokes $\{H_1, K_1\}, \dots, \{H_t, K_t\}$ in $D_2$ structured with respect to $(\nu, \kappa_2)$ such that
	\[t \ge \frac{1}{M^2} \Big \lfloor \frac{|\nu|}{M} \Big \rfloor - 3(n+1)M \ge \frac{1}{M^2} \Big \lfloor \frac{  f_{\ell}(r)}{M} \Big \rfloor - \frac{3 f_{\ell}(r)}{50M^3} \ge \frac{f_\ell(r)}{2M^3} \ge 34f_{\ell+1}(r) = 34c\]
	and for each $1 \le i \le t$, $\{H_i, K_i\}$ is in $\mathcal{R}(\mu_j)$ for some distinct $0 \le j \le n$ .

	For each $33c < i \le 34c$, 
	let $h_i$ (resp. $k_i$) be a simple path in the carrier $N(H_i) \subset D_2$ (resp. $N(K_i) \subset D_2$) starting from the edge of $\mu_0 \dots \mu_{n}$ dual to $H_i$ (resp. $K_i$) and with endpoint on $\kappa_2$. Additionally, we choose $h_i$ (resp. $k_i$) so that it does not intersect $H_i$ (resp. $K_i$) and intersects $\kappa_2$ at exactly one vertex.
	As the dual curves $H_1, K_1, \dots, H_t, K_t$ are pairwise non-intersecting, these paths can additionally be chosen to be pairwise non-intersecting.
	By applying a series of surgeries to insert reduced paths in place of the $h_i$ and $k_i$ using Lemma \ref{lem:tracking paths}, we assume that $h_i$ and $k_i$ are reduced for each $33c < i \le 34c$. Note that, even after such replacements, we still have that these paths are pairwise non-intersecting and that any dual curve that intersects $h_i$ (resp. $k_i$) must also intersect $H_i$ (resp. $K_i$).
	By a slight abuse of notation, we still denote this resulting diagram by $D$.
	Let $z_i$ be the subpath of $\mu_0 \dots \mu_{n}$ from the starting point of $h_i$ to that of $k_i$. By construction, $z_i \subset \mu_j$ for some distinct $0 \le j \le n$, $|z_i| \le |V(\Gamma)|$ (as our structured sequence of dual curves is tight) and the path $h_i^{-1} z_i k_i$ is a simple path with endpoints on $\kappa_2$.

	For each $33c < i \le 34c$, let $D_i \subset D$ be the subdiagram with boundary path $h_i \alpha_i k_i^{-1} z_i^{-1}$, where $\alpha_i$
	is the subpath of $\alpha$ between the endpoint of $h_i$ and the endpoint of $k_i$.
	Note that $\alpha_i \cap \alpha_j = \emptyset$ for all $i \neq j$.
	Let $x_i$ be the $(\ell-1)$--splitting point of $(h_i, k_i; z_i, \alpha_i)$ in $D_i$.
	Up to applying surgeries to $D_i$, we can assume that no diagram obtained from $D_i$ by a sequence of surgeries is such that the $(\ell-1)$--splitting point of $(h_i, k_i; z_i, \alpha_i)$ in this diagram occurs after $x_i$ along the orientation of $h_i$. 
	In other words, up to surgeries on $D_i$, $x_i$ is as far along $h_i$ as possible.
	Let $h_i'$ and $h_i''$ be the initial and terminal paths of $h_i$ with respect to $x_i$.
	There are now two main cases to consider, depending on whether or not $|h_i'|$ is small for all $33c < i \le 34c$. We will prove that Proposition \ref{prop:main_pieces} holds with (Y6a) in the first case and with (Y6b) in the second.
	\\\\
	\textbf{Case 1:}
	Suppose first that $|h_i'|  \le \frac{f_{\ell-1}(r)}{8}$ for all $33c < i \le 34c$.
	
	Let $e$ be the edge of $\mu_0 \dots \mu_n$ which is dual to $H_{33c}$.
	Let $0 \le m \le n$ be such that $e$ lies on $\mu_m$.
	Let $\rho$ be the initial subpath of $\psi \mu_0 \dots \mu_m$ up to and including $e$.
	We apply Lemma \ref{lem:tracking_paths2} to replace $\rho$ with a path $\rho'$, and by a slight abuse of notation, we denote by $H_{33c}$ the dual curve in this resulting diagram dual to a copy of the edge $e$.
	Let $\rho''$ be the initial subpath of $\rho'$ up to, and not including, the edge dual to $H_{33c}$.
	As in the proof above of Proposition \ref{prop:main_full} for the case $L=0$, we have that  $\rho'' = \hat{\psi} \hat{\mu}_0 \dots \hat{\mu}_m$ and that $|\hat{\mu}_0 \dots \hat{\mu}_m| \ge 32c$.
	We also have that no dual curve dual to $\rho''$ intersects $H_{33c}$.

	We would like to show that the conclusions of Proposition \ref{prop:main_pieces} hold by setting $\psi' = \hat{\psi}$, 
	setting $\mu_i' = \hat{\mu}_i$ for $1 \le i \le m$, 
	setting $\mu_i'$ to be the endpoint of $\hat{\mu}_m$ for $m < i \le n+2$ (i.e., a length $0$ path), 
	setting $\mathcal F' = \mathcal F$, and defining $\theta'$ to be a simple subpath of $N(H_{33c})$ with starting point the endpoint of $\hat{\mu}_m$ and endpoint on $\kappa_2$.
	Indeed, conclusions (Y1) and (Y2) immediately follow from our choices, and (Y3) holds from hypothesis (X3) as, for each $1 \le i \le n+2$, $\mu_i'$ either tracks $\mu_i$ or has length $0$. 
	Conclusion (Y4) holds from hypothesis (X4), as $\mu_1' \dots \mu_{n+2}'$ tracks $\mu_1 \dots \mu_n$. Conclusion (Y5) holds trivially as $|\mu_{n+1}'\mu_{n+2}'| = 0$. 
	Additionally, we have that $|\mu_0' \dots \mu_{n+2}'| = |\hat{\mu}_0 \dots \hat{\mu}_m| \ge 32c =  32f_{\ell+1}(R)$, so the first part of (Y6a) holds. 
	The remainder of this case consists of showing the bound  from  (Y6a) on the subpath $\alpha' \subset \alpha$.
	
	For $33c < i \le 34c$, set $\gamma_1^i$ to be the initial subpath of $h_i$ up to the starting point of $h_i''$. 
	By the definition of splitting points, $\gamma_1^i$ is just $h_i'$ with possibly the addition of an edge.
	Set $\gamma_2^i = h_i''$ and $\gamma_3^i$ to be the endpoint of $h_i''$ (i.e., a length $0$ path). 
	We get the following equations for $r$ large enough (depending only on $M$ and $\ell$), the first of which follows from (X1).
	\begin{align}
		&|\gamma_2^i| = |h_i''| \ge r - |\psi| - |\nu| - |h_i'| - 1\ge r - 2f_\ell(r) - \frac{f_{\ell-1}(r)}{8} - 1  \ge f_{\ell -1} \big(\frac{r}{4} \big ) \label{eq1} \\
		&r - |\psi| - |\nu| \ge r - 2f_\ell(r) \ge \frac{r}{4}  \label{eq2} \\
		&|\gamma_i^1| \le |h_i'| +1 \le \frac{f_{\ell - 1}(r)}{8} + 1 \le  f_{\ell -1 }(\frac{r}{4}) \label{eq_2} 
	\end{align}
	
	We now show that the hypotheses of Proposition \ref{prop:main_full} hold for the disk diagram $D_i$ with boundary path $\gamma_i^1 \gamma_i^2 \gamma_i^3 \alpha_i k_i^{-1} z_i^{-1}$ for $L = \ell - 1$ and $R = \frac{r}{4}$ (where in that proposition we set $\gamma_1 = \gamma_i^1$, $\gamma_2 = \gamma_i^2$, $\gamma_3 = \gamma_i^3$, $\alpha = \alpha_i$, $\eta = k_i^{-1}$ and $\beta = z_i^{-1}$).
	First note that as $|z_i| \le |V(\Gamma)|$, the same $M$ can be used in that proposition as the one used here. 
	Hypothesis (A1) of that proposition holds, as by equation (\ref{eq2}) and by (X1), the path $\alpha_i$ is $\frac{r}{4}$--avoidant with respect to $D_i$.
	We now check (A2).
	The path $\gamma_1^i \gamma_2^i$ is reduced as it is equal to the reduced path $h_i$. 
	No dual curve intersects $\gamma_3^i$ as it has length $0$.
	That $|\gamma_2^i| \ge f_{\ell - 1}(\frac{r}{4})$ and $|\gamma_1^i| \le f_{\ell - 1 }(\frac{r}{4})$ follow respectively from equations (\ref{eq1}) and~(\ref{eq_2}).

	By Remark \ref{rmk:splitting_points}, no $(\ell-1)$--fence connects $h_i''$ to $k_i$ in $D_i$. 
	Furthermore, by our choice of $D_i$ (with $x_i$ as furthest as possible along $h_i$) this is still true after performing surgeries to $D_i$. 
	Thus, every spoke which intersects $\gamma_2^i$ is $(\ell -1)$--fence separated from $k_i^{-1}$, and hypothesis~(A3) follows.

	As all the required hypotheses hold, we apply Proposition \ref{prop:main_full} and deduce, from conclusion (B4) of that proposition, that $|\alpha_i| \ge C_{\ell-1} (\frac{r}{4})^{\ell}$ for each $33c < i \le 34c$. 
	As the paths $\{\alpha_i\}$ are disjoint, we have that 
	\[ |\alpha'| \ge \sum_{i={33c+1}}^{34c} |\alpha_i| \ge c \Big(C_{\ell-1} \big(\frac{r}{4} \big)^{\ell} \Big) = f_{\ell +1} (r) \Big(C_{\ell-1} (\frac{r}{4})^{\ell} \Big)  \ge C_{\ell} r^{\ell+1}\]
	for some constant $C_\ell$ depending only on $M$ and $\ell$.
	Thus (Y6a) holds, and we are done in this case.
	\\\\
	\textbf{Case 2:} 
	By the previous case, we may assume that $|h_j'| > \frac{f_{\ell-1}(r)}{8}$ for some $33c < j \le 34c$. 
	We fix such a $j$. Let $0 \le m \le n$ be such that $\{H_j, K_j\}$ intersects $\mu_m$.
	Let $v$ be the starting point of $h_j$.
	Let $\omega$ be the initial subpath of $\psi \mu_0 \dots \mu_m$ up to $v$.

	By Lemma \ref{lem:tracking paths}, we can apply surgery to $\omega h_j$ to obtain a reduced path $\hat{\omega}\hat{h}_j = \hat{\psi} \hat{\mu}_0 \dots \hat{\mu}_m \hat{h}_j$ which tracks $\omega h_j$ and such that $\hat{h}_j = \hat{h}_j' \hat{a} \hat{h}_j''$ where $a$ is edge between the endpoint of $h_j'$ and the starting point of $h_j''$ as in the definition of splitting points.
	As $\omega$ and $h_j'$ are reduced paths, by the triangle inequality
	we have that 
	\begin{align}
		|\hat{h}_j'| \ge |h_j'| - |\psi| - |\nu| > \frac{f_{\ell - 1}(r)}{8} - 2f_{\ell}(r)  \ge f_{\ell - 1} \big(\frac{r}{16} \big)  
		\label{eq3}
	\end{align}
	Additionally, we have:
	\begin{align}
		|\hat{\psi} \hat{\mu}_1 \dots \hat{\mu}_m| &\le |\psi| + |\nu| \le 2f_{\ell}(r) \le f_{\ell - 1} \big(\frac{r}{16} \big) 
		\label{eq4}
	\end{align}

	Set $\gamma_1 = \hat{\psi} \hat{\mu}_0 \dots \hat{\mu}_m$. 
	Let $\gamma_2$ be the initial subpath of $\hat{h}_j'$ of length $\lceil f_{\ell - 1}(\frac{r}{16}) \rceil$, which exists by equation (\ref{eq3}) above.
	Let $\gamma_3$ be the subpath of $\hat{h}_j$ from the endpoint of $\gamma_2$ to the endpoint of $\hat{h}_j$.
	Let $\alpha''$ be the subpath of $\alpha$ from the endpoint of $\gamma_3$ to the endpoint of $\alpha$.

	We would now like to apply Proposition \ref{prop:main_full} to the disk diagram $D' \subset D$ with boundary path $\gamma_1 \gamma_2 \gamma_3 \alpha'' \eta \beta$ with $L = \ell - 1$ and $R = \frac{r}{16}$. 
	We first check that the hypotheses of that proposition hold.
	The path $\alpha''$ is $\frac{r}{16}$--avoidant with respect to $D'$, as it is a subpath of $\alpha$ which is $r$--avoidant with respect to $D$. Thus, (A1) holds.
	We now check (A2).
	As $\gamma_1 \gamma_2 \gamma_3 = \hat{\omega} \hat{h}_j$ is reduced, so is the path $\gamma_1 \gamma_2$,  and consequently no dual curve intersects both $\gamma_1 \gamma_2$ and $\gamma_3$.
	The bounds on $\gamma_1$ and $\gamma_2$ follow respectively from equation (\ref{eq4}) and our choice of $\gamma_2$.
	
	Let $\mathcal{F}'$ be the maximal $\ell$--fence in $D \setminus D'$ which contains $\{H_j, K_j\}$.
	If $m=0$, then $\{H_j, K_j\}$ is in $\mathcal{R}_0$ and so $\mathcal{F}'$ contains a spoke in $\mathcal{R}_0$.
	On the other hand, if $m >0$ then $\{H_j, K_j\}$ is in $\mathcal{R}_m$ and so $\mathcal{F}'$ contains $\mathcal{F}$ by (X3) and the maximality of $\mathcal{F}'$.
	In either case, $\mathcal{F}'$ must contain a spoke dual to an edge-pair that is  $\ell$--fence separated from $\eta$.
	
	Hypothesis (A3) now follows from the second statement of Claim \ref{claim:main_prop} below. To see this, let $\{e,f\}$ be an edge-pair in $\mathcal{R}(\gamma_2)$ where $\mathcal R (\gamma_2)$ is as in the claim. 
	Suppose for a contradiction that, after applying a series of surgeries to $D'$, in the resulting diagram there is a spoke $\mathcal{S}$ dual to $\{e,f\}$ that is contained in an $(\ell-1)$--fence $\mathcal{Z}$ which intersects $\eta$. 
	As $\mathcal{S}$ is in $\mathcal{R}(\gamma_2)$, it is also contained in $\mathcal{F}'$.
	By Lemma \ref{lem:combining_fences}, $\mathcal{F}' \cup \mathcal{Z}$ is an $\ell$--fence. This $\ell$--fence intersects $\eta$ as it contains $\mathcal Z$, and it contains a spoke that is dual to an edge-pair that is $\ell$--fence separated from $\eta$ (as $\mathcal{Z}$ does). This is a contradiction, and consequently (A3) follows.

\begin{figure}[h]
	\centering
		\begin{overpic}[scale=.5]{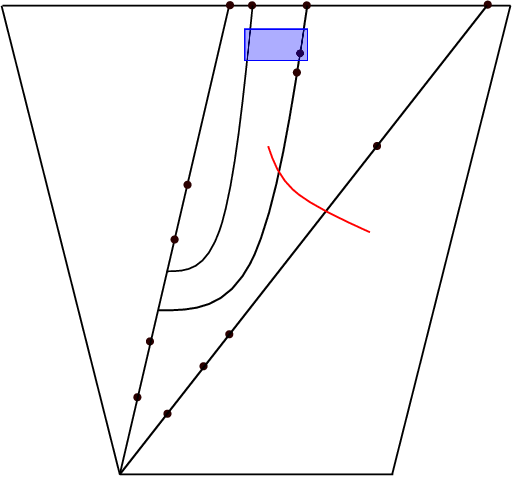}
			\put(22,9){\Tiny $\psi$}
			\put(21,20){\Tiny $\mu_0$}
			\put(23,30){\Tiny $\mu_1$}
			\put(32,36){\Tiny $z_j$}
			\put(28,51){\Tiny $\mu_2$}

			\put(48,72){\Tiny $k_j$}			
			\put(57,68){\Tiny $h_j'$}
			\put(60,87){\Tiny $h_j''$}
			\put(51,94){\Tiny $\alpha_j$}

			\put(29,4){\Tiny $\hat \psi$}
			\put(37,14){\Tiny $\hat \mu_0$}
			\put(43,22){\Tiny $\hat \mu_1$}
			
			\put(52,82){\Tiny$\mathcal{G}$}
			\put(69,49){\Tiny $Q$}	
			\put(53,35){\Tiny $\gamma_2$}
			\put(83,73){\Tiny $\gamma_3$}
		\end{overpic}
	\caption{Proof of Claim~\ref{claim:main_prop} when $m=1$.}
	\end{figure}
	\begin{claim} \label{claim:main_prop}
		A dual curve intersecting $\gamma_2$ must either intersect $\mathcal{F}'$ or intersect $z_j$.
		Additionally, the set $\mathcal{R}(\gamma_2)$ of edge-pairs dual to a spoke contained in $\mathcal{F}'$ is $M$--adequate. 
	\end{claim}
	\begin{proof}
		Set $\mathcal T := \{H_j, K_j\}$.
		By Remark \ref{rmk:splitting_points}, there exists an $(\ell-1)$--fence $\mathcal{G}$ connecting $h_j \setminus h_j'$ to $k_j$. 
		As every dual curve which intersects $h_j$ (resp. $k_j$) must intersect $H_j$ (resp. $K_j$), it follows that $\mathcal T \cup \mathcal{G}$ is an $\ell$--fence. 
		By Lemma \ref{lem:combining_fences} and as $\mathcal F'$ is maximal, $\mathcal T \cup \mathcal{G} \subset \mathcal F'$.
	
		To show the first claim, consider a dual curve $Q$ intersecting $\gamma_2$. 
		As $\gamma_2 \subset \hat{h}_j'$ and as $\hat{h}_j'$ tracks $h_j'$, $Q$ intersects $h_j'$.
		Furthermore, as $h_j$ is a reduced path, by considering $Q$ as a dual curve in $D_j$, one sees that $Q$ must intersect either $k_j$, $z_j$ or $\alpha_j$. 
		If $Q$ intersects $k_j$, then it intersects $\mathcal T$ and so intersects $\mathcal{F}'$. On the other hand, if $Q$ intersects $\alpha_j$, then by Proposition \ref{prop:intersects_a_spoke}, it intersects $\mathcal{G}$ and so intersects $\mathcal{F}'$. 
		The first claim now follows.
		
		To prove the second claim, consider a sequence $\{X_1, Y_1\}, \dots, \{X_q, Y_q\}$ of spokes structured with respect to $\gamma_2$.
		We must show that all but possibly $M$ of these spokes are contained in $\mathcal{F}'$.
		As $\hat{h}_j'$ tracks $h_j'$, these spokes are structured with respect to $h_j'$.
		As $|z_j| \le |V(\Gamma)|$ and as $M > |V(\Gamma)| + 1$, all but possibly $M$ of the spokes $\{X_i, Y_i\}$ intersect either $\alpha_j$ or $k_j$.
		If $\{X_i, Y_i\}$ intersects $k_j$, then it intersects $\mathcal T$.
		In this case, $\{\{X_i, Y_i\}, \mathcal T\}$ is a $1$--fence which contains a spoke of $\mathcal{F}'$ and so $\{X_i, Y_i\}$ is in $\mathcal{F}'$ by Lemma \ref{lem:combining_fences}.
		On the other hand, if $\{X_i, Y_i\}$ intersects $\alpha_j$, then it intersects $\mathcal{G}$ by Proposition \ref{prop:intersects_a_spoke}. In this case, $\{\{X_i, Y_i\}, \mathcal{G}\}$ is an $\ell$--fence containing a spoke of $\mathcal{F}'$ and again by Lemma \ref{lem:combining_fences} we deduce that $\{X_i, Y_i\}$ is contained in $\mathcal{F}'$.
	\end{proof}

	As the appropriate hypotheses are satisfied, we can now apply Proposition \ref{prop:main_full} to obtain a path $\gamma_1' \gamma_2' \zeta \pi$ in $D'$ from $b$ to $\alpha''$ satisfying the conclusions of that proposition for $L = \ell - 1$ and $R = \frac{r}{16}$. As $\gamma_1'$ tracks $\gamma_1$ by (B1), we have that $\gamma_1' = \hat{\hat{\psi}} \hat{\hat{\mu}}_0 \dots \hat{\hat{\mu}}_m$, and as $\gamma_2'$ tracks $\gamma_2$, we have that~$\gamma_2' = \hat{\hat{h}}_j'$.
	
	We now set $\psi' = \hat{\hat{\psi}}$ and $\mu_i' = \hat{\hat{\mu}}_i$ for each $0 \le i \le m$.
	For all $m < i \le n$ we set $\mu_i'$ to be the endpoint of $\mu_{m}'$.
	Let $\sigma$ be the initial subpath of $\gamma_2' \zeta$ of length $\lceil f_\ell(r) - |\mu_0' \dots \mu_{n}'| \rceil$ which exists by (B1) as $|\gamma_2' \zeta| \ge 32 f_{\ell}(\frac{r}{16}) = 2 f_{\ell}(r)$.
	If the endpoint of $\gamma_2'$ does not lie on $\sigma$, we set $\mu_{n+1}'$ to be equal to $\sigma$ and $\mu_{n+2}'$ to be the endpoint of $\sigma$. Otherwise, we set $\mu_{n+1}'$ to be equal to $\gamma_2'$ and $\mu_{n+2}'$ to be the subpath of $\zeta$ from the endpoint of $\gamma_2'$ to the endpoint of $\sigma$.
	Finally, we set $\theta'$ to be the subpath of $\gamma_1'\gamma_2' \zeta \pi$ from the endpoint of $\mu_{n+2}'$ to the endpoint of $\pi$. 
	
	To conclude the proof, we now check that the conclusions of Proposition \ref{prop:main_pieces} are satisfied with these choices. 
	Conclusion (Y1) immediately follows by our choice of $\mathcal{F}'$. 
	We now check (Y2).
	The path $\psi' \mu_0' \dots \mu_{n+2}'$ is reduced as it is a subpath of $\gamma_1'\gamma_2'\zeta$ which is reduced by (B1). 
	No dual curve intersects both $\mu_0' \dots \mu_{n+2}'$ and $\theta'$ by (B1). 
	The paths $\psi', \mu_0' \dots, \psi_n'$ track the paths $\psi, \mu_0, \dots, \mu_n$ respectively by construction. 
	Consequently, conclusion (Y2) follows.
	
	We now check conclusion (Y3).
	Given $1 \le i \le m$, (Y3) holds for the spokes intersecting $\mu_i'$ by (X3), as this path tracks $\mu_i$.
	Given $m < i \le n$, (Y3) trivially holds for spokes intersecting $\mu_i'$ as this is a length $0$ path.
	Conclusion (Y3) holds for $\mu_{n+1}' = \gamma_2'$ by the second part of Claim \ref{claim:main_prop} and as $\gamma_2'$ tracks $\gamma_2$. 
	Finally, suppose that we have a sequence of spokes structured with respect to $\mu_{n+2}' \subset \zeta$. 
	By (B2), all but possibly $M$ of these spokes are contained in an $\ell$--fence which contains a spoke in $\mathcal{R}(\gamma_2)$ (where $\mathcal{R}(\gamma_2)$ is as in Claim~\ref{claim:main_prop}).
	Thus, by Lemma \ref{lem:combining_fences} and as $\mathcal{F}'$ is maximal, all but possibly $M$ of these spokes are contained in $\mathcal{F}'$.
	Thus,  (Y3) holds.
	
	Before checking (Y4), we prove the following claim:
	\begin{claim} \label{claim:main_prop2}
		Any dual curve which intersects $\mu_{n+2}'$, either intersects $\mathcal{F}'$ or intersects $z_j$.
	\end{claim}
	\begin{proof}
		Let $Q$ be a dual curve intersecting $\mu_{n+2}'$. 
		As $\mu_{n+2}' \subset \zeta$, it follows from (B3) that $Q$ either intersects $\gamma_2$ or intersects an $(\ell-1)$--fence that contains a spoke of $\mathcal R (\gamma_2)$.
		In the first case, we are done by the first part of Claim \ref{claim:main_prop}. 
		Otherwise, there exists an $(\ell-1)$--fence $\mathcal{G}$ that contains a spoke in $\mathcal R (\gamma_2)$ and which $Q$ intersects.
		As $\mathcal G$ contains a spoke of $\mathcal R (\gamma_2)$, by Lemma \ref{lem:combining_fences} every spoke of $\mathcal G$ is contained in $\mathcal F'$. 
		Thus, $Q$ intersects $\mathcal F'$ as claimed.
	\end{proof}

	We are now ready to check (Y4).
	Let $A_1, B_1, \dots, A_p, B_p$ be a sequence of dual curves structured with respect to $\mu_1' \dots \mu_{n+2}'$.
	By the first part of Claim \ref{claim:main_prop}, Claim \ref{claim:main_prop2} and as $\mu_i'$ tracks $\mu_i$ for $1 \le i \le n$, it follows that each $A_i$ (resp. $B_i$) either intersects $\mathcal{F}'$, intersects $\mu_1 \dots \mu_n$ or intersects $z_j$.
	If $m > 0$, as $z_j \subset \mu_m$, it follows that these dual curves either intersect $\mathcal{F}'$ or intersect $\mu_1 \dots \mu_n$ and the claim follows from (X4).
	On the other hand, if $m = 0$, then $\mu_{1}' \dots \mu_{n+2}' = \mu_{n+1}' \mu_{n+2}'$, and it follows from the first part of Claim \ref{claim:main_prop} and Claim \ref{claim:main_prop2} that a dual curve intersecting $\mu_{n+1}' \mu_{n+2}'$ either intersects $\mathcal{F}'$ or intersects $z_j \subset \mu_0$.
	As $|z_j| \le V(\Gamma) < M$, the claim also follows in this case.
	
	Conclusion (Y5) holds by the first part of Claim \ref{claim:main_prop} and Claim \ref{claim:main_prop2}, as $z_j \subset \mu_0 \dots \mu_n$, $\mu_{n+1}' \subset \gamma_2'$ and $\gamma_2'$ tracks $\gamma_2$. We now check that (Y6b) holds. The first part of this claim follows as $|\mu_0' \dots \mu_{n+2}'| = \lceil f_{\ell}(r) \rceil$ by our choices. The subpath of $\alpha$ between the endpoint of $\gamma_3$ and the endpoint of $\pi$ (which is the same as the endpoint of $\theta'$) has length $C_{\ell-1} (\frac{r}{16})^{\ell}$ by conclusion (B4) of Proposition \ref{prop:main_full}. 
	Thus, (Y6b) follows. This completes the proof for this final case.
\end{proof}

Before proving the final step, we show how Proposition \ref{prop:main_pieces} can naturally be iterated.

\begin{figure}[h]
	\centering
	\begin{overpic}[scale=.5]{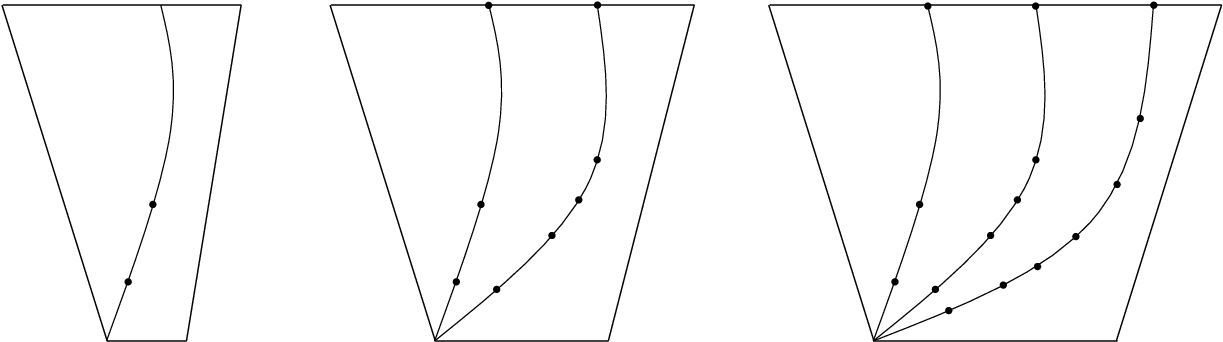}
		\put(8,3){\Tiny $\psi$}
		\put(9,8.5){\Tiny $\mu_0$}
		\put(12.5,20){\Tiny $\theta$}
		
		\put(23,11){$\to$}
		
		\put(35,3){\Tiny $\psi$}
		\put(35.5,8){\Tiny $\mu_0$}
		\put(39.5,20){\Tiny $\theta$}
		
		\put(37.5,3){\Tiny $\psi'$}
		\put(40.5,7){\Tiny $\mu_0'$}
		\put(43.5,10.5){\Tiny $\mu_1'$}
		\put(45,14){\Tiny $\mu_2'$}
		\put(47,20){\Tiny $\theta'$}
		\put(42,28){\Tiny $\approx R^L$}
		
		\put(58,11){$\to$}		
		
		\put(71,3){\Tiny $\psi$}
		\put(71.5,8){\Tiny $\mu_0$}
		\put(75.5,20){\Tiny $\theta$}

		\put(73.5,3){\Tiny $\psi'$}
		\put(76.5,7){\Tiny $\mu_0'$}
		\put(79.5,10.5){\Tiny $\mu_1'$}
		\put(81,14){\Tiny $\mu_2'$}
		\put(83,20){\Tiny $\theta'$}

		\put(75,.5){\Tiny $\psi''$}
		\put(79,2){\Tiny $\mu_0''$}
		\put(83,3.5){\Tiny $\mu_1''$}
		\put(86,6){\Tiny $\mu_2''$}
		\put(90,10){\Tiny $\mu_3''$}
		\put(92.5,15){\Tiny $\mu_4''$}
		\put(94,20){\Tiny $\theta''$}
		\put(78,28){\Tiny $\approx R^L$}
		\put(87,28){\Tiny $\approx R^L$}
	\end{overpic}
\caption{Illustration of Lemma~\ref{lem:iterate}. Each arrow symbolizes an application of Proposition~\ref{prop:main_pieces}.}
\end{figure}

\begin{lemma} \label{lem:iterate}
		Suppose we have integers $L \ge 1$, $M \ge 2$ and $R \ge 0$. Furthermore, suppose we have a disk diagram containing a path $\psi' \mu_0' \dots \mu_{m+2}' \theta'$, with $m +2 \le \frac{f_{L}(R)}{50M^3} - 1$, and an $L$--fence $\mathcal{F}'$ 
		which were obtained by applying Proposition \ref{prop:main_pieces} (where $n = m$ in that proposition) and such that conclusion (Y6b) of that proposition holds. Then, by setting $\mathcal F = \mathcal F'$, $\psi = \psi'$, $\theta = \theta'$ and for $1 \le i \le m+2$ setting $\mu_i = \mu_i'$, it follows that the hypotheses of Proposition \ref{prop:main_pieces} hold for the path $\psi \mu_0 \dots \mu_{m+2} \theta$ (by setting $n = m +2$ in that proposition).
\end{lemma}
\begin{proof}
	First note that $\mathcal{F}'$ contains a spoke which is $L$--fence separated from $\eta$ by (Y1),
	 thus so does $\mathcal F$. Hypothesis (X1) follows, as it already was true before Proposition \ref{prop:main_pieces} was applied to obtain the path $\psi' \mu_0' \dots \mu_{m+2}' \theta'$.
	We now check (X2). From (Y2) we see that $\psi \mu_0 \dots \mu_{m+2}$ is reduced and  no dual curve intersects both $\psi \mu_0 \dots \mu_{m+2}$ and $\theta$. That $|\psi| \le f_{L}(R)$ follows by (Y2) as $\psi$ tracks a path of length at most $f_L(R)$. Finally, $|\mu_0 \dots \mu_{m+2}| \ge f_L(R)$ by (Y6b). Thus, (X2) holds.
	The claims (X3) and (X4) immediately follow from (Y3) and (Y4) respectively. Finally, (X5) follows as, by (Y2), $\mu_0$ tracks a path satisfying (X5).
\end{proof}

\begin{proof}[Proof of Proposition \ref{prop:main_full} for $L \ge 1$ assuming Proposition \ref{prop:main_pieces} for $L$.]
	Let $D$ be a disk diagram with boundary path $\gamma_1 \gamma_2 \gamma_3 \alpha \eta \beta$ satisfying the hypotheses of Proposition \ref{prop:main_full} for some $L \ge 1$, and suppose that Proposition \ref{prop:main_pieces} holds for $L$.
	We show that the conclusions of Proposition \ref{prop:main_full} hold for $D$.
	Our strategy will be to use Lemma \ref{lem:iterate} to iteratively apply Proposition \ref{prop:main_pieces}. 
	At the $i$'th iteration, we show that we have a path satisfying conclusions (B1)--(B3) of Proposition \ref{prop:main_full}. Furthermore, we additionally show that either conclusion (B4) holds (and we are done) or, alternatively, the subpath of $\alpha$, as in conclusion (B4), has length at least $C_L' \frac{n}{2} R^L$ where $C_L'$ is the constant from Proposition \ref{prop:main_pieces} and $n = 2i$.
	This is enough to prove the proposition, as when $n \ge \frac{f_L(R)}{50M^4}$, we get that $|\alpha| \ge C_L' \frac{f_L(R)}{100M^4} R^L \ge C R^{L+1}$, where $C$ depends only on $M$ and~$L$, and so (B4) follows.
	
	First we show there exists a path $\psi \mu_0 \theta$ in $D$ satisfying the hypotheses of Proposition \ref{prop:main_pieces} with $n=0$. This is seen by setting $\psi = \gamma_1$, $\mu_0 = \gamma_2$ and $\theta = \gamma_3$. 
	We also take $\mathcal{F}$ to be any spoke intersecting $\gamma_2$ that is contained in $\mathcal{R}(\gamma_2)$, which exists for $R$ large enough by Lemma \ref{lem:nice_sequence}.
	The hypotheses (X1), (X2) and (X5) of Proposition \ref{prop:main_pieces} follow respectively from hypotheses (A1), (A2) and (A3) of Proposition \ref{prop:main_full} (where we set $\mathcal R (\mu_0) = \mathcal R (\gamma_2)$).
	Furthermore, hypotheses	(X3) and (X4) follow trivially.
	
	Next, suppose that we have a path $\psi \mu_0 \dots \mu_n \theta$ in $D$ and an $L$--fence $\mathcal{F}$ satisfying the hypotheses of Proposition \ref{prop:main_pieces} (where $\psi$, $\mu_0$, $\theta$ and $\mathcal{F}$ are possibly different than as in the previous paragraph). Additionally, suppose that $\mathcal{F}$ contains a spoke of $\mathcal{R}(\gamma_2)$, that $\mu_0$ tracks $\gamma_2$ and that $\psi$ tracks $\gamma_1$.
	Furthermore, suppose that every dual curve which intersects $\mu_0 \dots \mu_n$ either intersects $\gamma_2$ or intersects an $L$--fence which contains a spoke of $\mathcal R (\gamma_2)$.
	Finally, we suppose that the subpath of $\alpha$ from the endpoint of $\gamma_3$ to the endpoint of $\theta$ has length $C_L' \frac{n}{2} R^L$ where $C'$ depends only on $M$ and $L$. 
	
	We apply Proposition \ref{prop:main_pieces} to obtain a new path $\psi' \mu_0' \dots \mu_{n+2}' \theta'$ satisfying the conclusions of that proposition. 
	When applying this proposition, we set $\mathcal R (\mu_0)$ 
	to be all edge-pairs $\{e,f\}$ in $\mu_0$ such that there exists a spoke dual to $\{e,f\}$ and in 
	$\mathcal R (\gamma_2)$. 
	Note that the set $\mathcal R (\mu_0)$ is $M$--adequate as $\mathcal R (\gamma_2)$ is $M$--adequate and $\mu_0$ tracks $\gamma_2$. 
	Let $\mathcal F'$ be as in (Y1), and note that $\mathcal F'$ contains a spoke in $\mathcal R (\gamma_2)$  by our assumption on $\mathcal F$.
	
	We now show that conclusions (B1)--(B3) of Proposition \ref{prop:main_full} hold by setting $\gamma_1' = \psi'$, $\gamma_2' = \mu_0'$, $\zeta = \mu_1' \dots \mu_{n+2}'$ and $\pi = \theta'$.
	We first check that (B1) holds. 
	By (Y2) we get that $\gamma_1'\gamma_2'\zeta$ is reduced, no dual curve intersects both $\gamma_1'\gamma_2'\zeta$ and $\pi$, 
	$\gamma_1'$ tracks $\psi$ (and consequently tracks $\gamma_1$) and $\gamma_2'$ tracks $\mu_0$ (and consequently tracks $\gamma_2$).
	If (Y6a) holds then we immediately get that $|\gamma_2' \zeta| \ge 32 f_{L+1}(R)$.
	On the other hand, if (Y6b) holds, then $|\gamma_2' \zeta| = |\mu_0' \dots \mu_{n+2}'| \ge f_L(R) \ge 32 f_{L+1}(R)$.
	Conclusion (B1) now follows.
	
	Conclusion (B2) follows from (Y4). 	
	Finally, to check (B3), let $Q$ be a dual curve that intersects $\zeta = \mu_1' \dots \mu_{n+2}'$.
	If $Q$ intersects $\mu_1' \dots \mu_n'$, then it intersects $\mu_1 \dots \mu_n$ by (Y2).
	On the other hand, if $Q$ intersects $\mu_{n+1}'\mu_{n+2}'$ then by (Y5) it either intersects $\mathcal F'$ or $\mu_0 \dots \mu_n$. 
	Thus, $Q$ either intersects $\mathcal F'$ or intersects $\mu_0 \dots \mu_n$.
	In the first case, (B3) immediately follows as $\mathcal F'$ contains a spoke of $\mathcal R (\gamma_2)$. In the latter case, (B3) follows as, by our assumption, every dual curve which intersects $\mu_0 \dots \mu_n$ either intersects $\mathcal F$ or intersects $\gamma_2$.
	
	We have thus established that (B1)--(B3) hold. We now check that our additional assumptions on $\mu_0 \dots \mu_n$ correspondingly hold for the path $\mu_0' \dots \mu_{n+2}'$.
	We have already established that $\mathcal F'$ contains a spoke of $\mathcal R (\gamma_2)$, $\psi'$ tracks $\gamma_1$ and $\mu_0'$ tracks $\gamma_2$. Additionally, by arguing as in the previous paragraph, every dual curve which intersects $\mu_0' \dots \mu_{n+2}'$ intersects either $\gamma_2$ or an $L$--fence which contains a spoke of $\mathcal R (\gamma_2)$. 

	Now, if (Y6a) holds, then (B4) follows and we are done.
	On the other hand, if (Y6b) holds, it then follows that the subpath of $\alpha$ between the endpoint of $\theta$ and the endpoint of $\theta'$ has length at least $C_L' R^L$. Thus, the subpath of $\alpha$ between the endpoint of $\gamma_3$ and the endpoint of $\theta'$ has length at least $C_L'\frac{n}{2} R^L + C_L' R^L = C_L'\frac{n+2}{2} R^L$.
	Thus, the required hypotheses hold. Additionally, if $n+2 \le \frac{f_{L}(R)}{50M^4} - 1$, then by Lemma \ref{lem:iterate} we may iterate and apply Proposition \ref{prop:main_pieces} again.
	This completes the proof.
\end{proof}

\section{Main theorems} \label{sec:main_thms}
\emph{In this section we use Theorem~\ref{thm:main} to prove the results from the introduction.}

\vspace{.4cm}
Before proving Theorem \ref{intro_thm:div}, we first define $\Gamma$--complete words.
These words were first defined in \cite{Dani-Thomas} and are also utilized in \cite{Levcovitz-div}.
The periodic geodesic we construct to prove Theorem~\ref{intro_thm:div} will have its label be the concatenation of $\Gamma$--complete words.
Recall that given a graph $\Gamma$, the \emph{complement graph $\Gamma^c$} is the graph with the same vertex set as $\Gamma$ and where distinct vertices of $\Gamma^c$ are adjacent if and only if they are not adjacent in $\Gamma$.

\begin{definition}[$\Gamma$-complete word] \label{def:complete_word}
	Let $\Gamma$ be a graph which is not a join.
	As $\Gamma$ is not a join, the
	complement graph $\Gamma^c$ is connected and consequently there is 
	a sequence of vertices $s_0, \dots, s_n$ of $\Gamma$ such that 
	\begin{enumerate}
		\item For every vertex $s \in \Gamma$, $s_i = s$ for some $0 \le i \le n$.
		\item The vertices $s_i$ and $s_{i+1}$ are distinct, non-adjacent vertices of
		$\Gamma$ for all $0 \le i \le n$ (taken mod $n+1$).
	\end{enumerate}
	Note that it could be
	that $s_i = s_j$ for some $i \neq j$.
	We say that $w = s_0 \dots s_n$ is
	a \textit{$\Gamma$-complete word}.
		We remark that by Tits' solution to the word problem (see \cite{Davis} for instance), $w^n$ is a reduced word for all integers $n$.
\end{definition}

We deduce the next lemma from Propositions \ref{prop:hyp_index_of_fences} and \ref{prop:intersects_a_spoke}.
\begin{lemma} \label{lem:complete_word}
	Let $\Gamma$ be a graph that is not a join and with hypergraph index $L \not\in \{0, \infty\}$.
	Let $D$ be a disk diagram with simple boundary path $\gamma \alpha \gamma' \beta$.
	Suppose that the label of $\beta$ is a $\Gamma$-complete word and that every dual curve dual to $\beta$ is also dual to
	$\alpha$.
	Then $\gamma$ and $\gamma'$ are not connected by an $(L-1)$--fence.
\end{lemma}
\begin{proof}
	Suppose for a contradiction that there is an $(L-1)$--fence $\mathcal{F}$ connecting
	$\gamma$ and $\gamma'$.
	Let $w = s_0 \dots s_n$ be the $\Gamma$--complete word which is the label of $\beta$.
	By Proposition \ref{prop:intersects_a_spoke}, every dual curve dual to $\beta$
	intersects a spoke of $\mathcal{F}$. 
	By Proposition \ref{prop:hyp_index_of_fences}, the set of vertices $V(\mathcal{F}) \cup \{s_0,
	,\dots, s_n \}$ either induces a strip subgraph of $\Gamma$ or induces a subgraph of hypergraph index at most $L-1$. 
	However, as $w$ is a $\Gamma$--complete word, every $s \in V(\Gamma)$ is equal to $s_i$
	for some $i$. 
	This implies that $V(\Gamma) = V(\mathcal{F}) \cup \{s_0,
	,\dots, s_n \}$. Thus, $\Gamma$ is either a strip subgraph (and has hypergraph index $\infty$) or has hypergraph index at most $L-1$. 
	In either case we get a contradiction.
\end{proof}

We can now prove the theorems from the introduction.

\begin{theorem}[Theorem C]
	Let $\Gamma$ be a simplicial graph with hypergraph index $k \neq \infty$. Then the Cayley graph of the RACG $W_\Gamma$ contains a periodic geodesic with geodesic divergence a polynomial of degree~$k+1$.
\end{theorem}
\begin{proof}
	First note that we may assume that $\Gamma$ is not a clique, as otherwise $W_{\Gamma}$ is a finite group and has hypergraph index $\infty$. Suppose first that $\Gamma = \Gamma_1 \star \Gamma_2$ is a join.
	If $\Gamma_2$ is a clique, then $W_{\Gamma} = W_{\Gamma_1} \times W_{\Gamma_2}$ where $W_{\Gamma_2}$ is finite. 
	In this case, it can readily be deduced that $\Gamma_1$ and $\Gamma$ have the same hypergraph index
	 and that the geodesic divergence of a bi-infinite geodesic in the Cayley graph of $W_{\Gamma_1}$ is equivalent, under the $\asymp$ equivalence of functions, to the geodesic divergence of this geodesic when considered as a geodesic in the Cayley graph of $W_\Gamma$.
	Thus, we may assume that $\Gamma_1$ and $\Gamma_2$ each contain a pair of non-adjacent vertices. 
	In particular, $\Gamma$ has hypergraph index $0$. 
	Furthermore, $W_\Gamma$ is strongly thick of order $0$ and has linear divergence \cite{Behrstock-Falgas-Ravry-Hagen-Susse}[Proposition 2.11]. 
	In particular, any periodic geodesic in $W_\Gamma$ has geodesic divergence a linear function, and the claim follows in this case.
	
	By the previous paragraph, we may assume that $\Gamma$ is not a join graph, and we form a $\Gamma$--complete word  $w=s_0 \dots s_n$. 
	We further assume that $n$ is minimal out of the possible choices for $w$. In particular, $|w|$ only depends on $\Gamma$.
	Let $\Sigma_{\Gamma}$ be the Davis complex of the RACG $W_\Gamma$.
	Let $\sigma$ be the 
	bi-infinite geodesic based at the identity vertex $b \in \Sigma_{\Gamma}$ which has one of its infinite rays emanating from $b$ with label 
	 $www\dots$ and the other ray emanating from $b$ with label $w^{-1}w^{-1}w^{-1} \dots$.
	For $i \in \mathbb{Z}$, let $p_i$ be the vertex of $\sigma$ which is the endpoint of the subpath of $\sigma$ with starting point $b$   and label~$w^i$.
	
	Fix an integer $r > 0$. Let $B$ be the open ball of radius $|w|r$ based at $b \in \Sigma_\Gamma$, and let $\nu$ be a simple path in $\Sigma_\Gamma \setminus B$ from $p_r$ to $p_{-r}$.
	The path $\nu$ exists as $W_\Gamma$ is one-ended (since it has integer hypergraph index and consequently is not relatively hyperbolic).
	To prove the theorem, it is enough to show that, for $r$ large enough, the length of $\nu$ is bounded below by a function $Cr^{k+1}$ for some constant~$C$.
	
	\begin{figure}[htbp]
		\begin{overpic}[scale=.35]{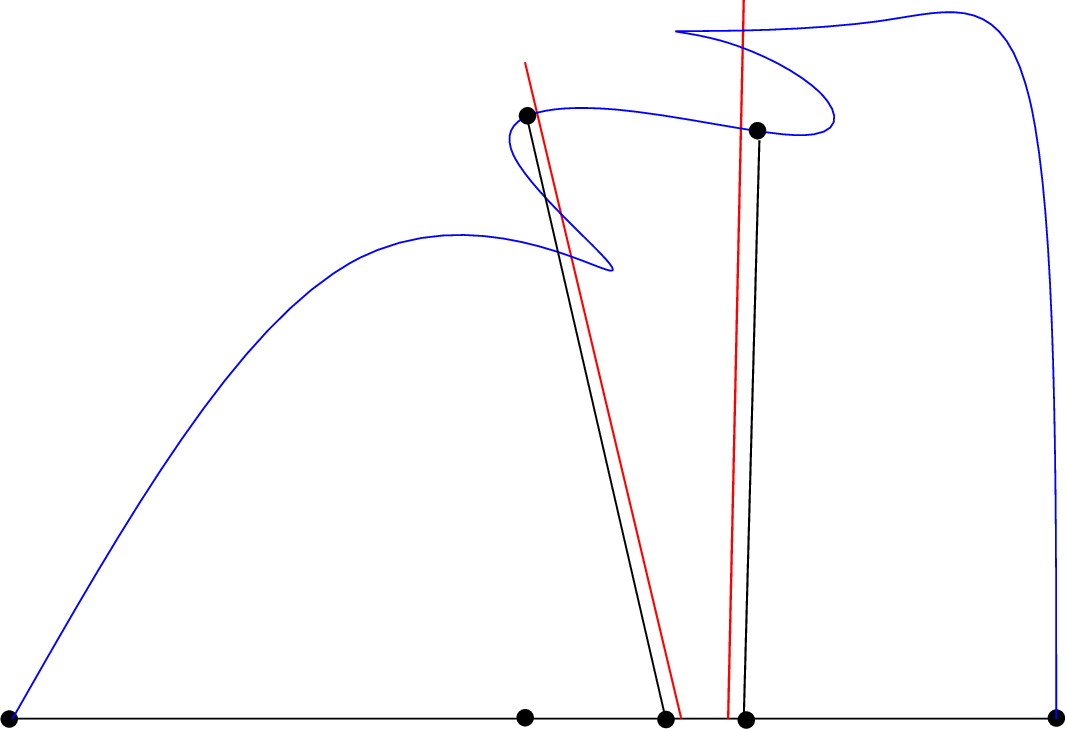}
			\put(48,-3){\Tiny$b$}
			\put(60,-2){\Tiny$p_i$}
			\put(69,-2){\Tiny$p_{i+1}$}
			\put(57,30){\Tiny$H_i$}
			\put(64,30){\Tiny$K_i$}
			\put(54,20){\Tiny$\gamma_i$}
			\put(71,20){\Tiny$\eta_i$}
			\put(58,59){\Tiny$\alpha_i$}
			\put(20,34){\Tiny$\nu$}
			\put(64.5,-3){\Tiny$\beta_i$}
		\end{overpic}
		\caption{Paths chosen in the proof of Theorem~C.}
		\label{fig:thm_c}
	\end{figure}
	
	For each $0 \le i < r$, let $H_i$ (resp. $K_i$) be the hyperplane dual to the edge of $\sigma$ which is adjacent to $p_i$ (resp. $p_{i+1}$) and is labeled by $s_0$ (resp. $s_n$).
	As hyperplanes separate $\Sigma_\Gamma$ and do not intersect geodesics twice, it follows that $H_i$ (resp. $K_i$) intersects $\nu$.
	Let $e_i$ (resp. $f_i$) be the last (resp. first) edge of $\nu$ dual to $H_i$ (resp. $K_i$).
	Let $\gamma_i$ (resp. $\eta_i$) be a minimal length geodesic in the carrier $N(H_i)$ (resp. $N(K_i)$) with starting point $p_i$ (resp. $p_{i+1}$) and endpoint on $e_i$ (resp. $f_i$). 
	Let $\alpha_i$ be the subpath of $\nu$ between $\gamma_i \cap \nu$ and $\eta_{i} \cap \nu$.
	As $w$ is a $\Gamma$--complete word, no pair of hyperplanes dual to $\sigma$ intersect.
	By our choices, $\alpha_i \cap \alpha_j$ is either empty or a single vertex for all $i \neq j$.
	Let $D_i$ be the disk diagram with boundary path $\gamma_i \alpha_i \eta_{i}^{-1} \beta_i$ where $\beta_i$ has label~$w^{-1}$.

	For each $0 \le i \le \frac{r}{2}$, we observe the following:
	\begin{enumerate}
		\item The path $\gamma_i$ is reduced.
		\item By Lemma \ref{lem:complete_word}, no $(k-1)$--fence connects $\gamma_i$ to $\eta_{i}^{-1}$ in \textit{any} disk diagram with boundary path $\gamma_i \alpha_i \eta_{i}^{-1} \beta_i$. 
		\item The path $\alpha_i$ does not intersect the ball $B_{p_i}(|w|(r - i)) = B_{p_i}(R)$, since $\nu$ does not intersect the ball $B_{p_0}(|w|r)$.
	\end{enumerate}
	Thus, for each $0 \le i \le \frac{r}{2}$, we can apply Theorem~\ref{prop:main_full} to $D_i$ by setting, in that theorem, $\gamma = \gamma_i$, $\alpha = \alpha_i$, $\eta = \eta_{i}^{-1}$, $\beta = \beta_i$, $L=k-1$ and $R = |w|(r - i)$. Hypotheses (1)--(3) in that theorem follow respectively from (1)--(3) above. We conclude that for $r$ large enough $|\alpha_i| \ge C' (|w|(r - i))^{k}$ where $C'$ depends only on $\Gamma$ and $k-1$. 
	
	As the $\{\alpha_i\}$ do not have edges in common, for $r$ large enough we get:
	\[|\nu| \ge \sum_{i=1}^{ \big \lfloor \frac{r}{2} \big \rfloor} |\alpha_i| \ge \Big(\frac{r}{2}-1 \Big) C' \Big(|w| \big(r - \frac{r}{2} \big) \Big)^{k} \ge Cr^{k+1}\]
	where $C$ depends only on $\Gamma$ and~$k$.
\end{proof}

\begin{theorem}[Theorem A]
	Let $W_\Gamma$ be a RACG and $k \ge 0$ an integer. Then the following are equivalent:
	\begin{enumerate}
		\item The hypergraph index of $\Gamma$ is $k$.
		\item The divergence of $W_\Gamma$ is $r^{k+1}$, and the Cayley graph of $W_\Gamma$ contains a periodic geodesic with geodesic divergence $r^{k+1}$.
		\item  The group $W_\Gamma$ is strongly thick of order $k$.
	\end{enumerate}
\end{theorem}  
\begin{proof}
	Suppose first that the hypergraph index of $\Gamma$ is $\infty$. 
	It then follows that $W_\Gamma$ is relatively hyperbolic (see \cite{Levcovitz-thick}). Thus, $W_\Gamma$ has either exponential divergence or infinite divergence \cite{Sisto}[Theorem~1.3], and $W_\Gamma$ is not strongly thick \cite{Behrstock-Drutu-Mosher}.
	
	Thus, we may suppose that $W_\Gamma$ has hypergraph index a non-negative integer $k$.
	Consequently, $W_\Gamma$ is strongly thick of order at most $k$ by Theorem \ref{thm:hi_bounds_thick} and 
	has divergence function bounded above by the function $r^{k+1}$ by Theorem \ref{thm_thick_bounds_div}.
	By Theorem \ref{intro_thm:div}, $W_\Gamma$ contains a periodic geodesic with geodesic divergence the function $r^{k+1}$ and, consequently, $W_\Gamma$ has divergence function bound below by $r^{k+1}$. Thus, $W_\Gamma$ has divergence exactly $r^{k+1}$. 
	Applying  Theorem \ref{thm_thick_bounds_div} once more, we see that $W_\Gamma$ is strongly thick of order exactly $k$.
\end{proof}

\bibliographystyle{amsalpha}
\bibliography{bibliography}
\end{document}